\documentclass[11pt, twoside, leqno]{article}

\usepackage{amssymb}
\usepackage{amsmath}
\usepackage{amsthm}
\usepackage{color}
\usepackage{mathrsfs}
\usepackage{indentfirst}

\usepackage{txfonts}

\allowdisplaybreaks	

\def\red{\color{red}}

\def\rr{{\mathbb R}}
\def\rn{{{\rr}^n}}
\def\zz{{\mathbb Z}}

\def\nn{{\mathbb N}}

\def\cb{{\mathcal B}}

\def\fz{\infty}
\def\az{\alpha}
\def\bz{\beta}

\def\bgz{{\Gamma}}

\def\lz{\lambda}

\def\tz{\theta}

\def\lf{\left}
\def\r{\right}

\def\hs{\hspace{0.25cm}}
\def\ls{\lesssim}

\def\noz{\nonumber}
\def\wz{\widetilde}

\def\st{\subset}
\def\com{\complement}

\def\supp{\mathop\mathrm{\,supp\,}}

\def\loc{\mathop\mathrm{\,loc\,}}

\def\dydt{\,\frac{dy\,dt}{t^{n+1}}}

\theoremstyle{definition}

\newtheorem{theorem}{Theorem}[section]

\newtheorem{lemma}[theorem]{Lemma}

\newtheorem{definition}[theorem]{Definition}

\newtheorem{remark}[theorem]{Remark}

\numberwithin{equation}{section}

\begin{document}

\title{\bf\Large Quantitative Boundedness of Littlewood--Paley Functions
on Weighted Lebesgue Spaces in the Schr\"{o}dinger Setting
\footnotetext{\hspace{-0.35cm} 2010 {\it
Mathematics Subject Classification}. Primary 42B25; Secondary 35J10, 42B20.
\endgraf {\it Key words and phrases}. Schr\"{o}dinger operator,
Muckenhoupt weight, Littlewood--Paley function, quantitative boundedness.
\endgraf Junqiang Zhang is supported by the National
Natural Science Foundation of China (Grant Nos. 11801555 and 11971058) and
the Fundamental Research Funds for the Central Universities
(Grant No. 2018QS01).
Dachun Yang is supported by the National
Natural Science Foundation of China (Grant Nos. 11971058, 11761131002
and 11671185).
}}
\author{Junqiang Zhang
and Dachun Yang\,\footnote{Corresponding author,
E-mail: \texttt{dcyang@bnu.edu.cn}/{\red August 29, 2019}/Final version.}}
\date{}
\maketitle

\vspace{-0.8cm}

\begin{center}
\begin{minipage}{13.8cm}
{\small {\bf Abstract}\quad
Let $L:=-\Delta+V$ be the Schr\"{o}dinger
operator on $\mathbb{R}^n$ with $n\geq 3$,
where $V$ is a non-negative potential which belongs to certain
reverse H\"{o}lder class $RH_q(\mathbb{R}^n)$ with $q\in (n/2,\,\infty)$.
In this article, the authors obtain the quantitative weighted boundedness of Littlewood--Paley functions
$g_L$, $S_L$ and $g_{L,\,\lambda}^\ast$, associated to $L$,
on weighted Lebesgue spaces $L^p(w)$, where
$w$ belongs to the class of Muckenhoupt $A_p$ weights adapted to $L$.
}
\end{minipage}
\end{center}


\section{Introduction}\label{s1}

The class of Muckenhoupt weights was originally introduced by Muckenhoupt \cite{m72} in 1972.
Recall that a non-negative locally integrable function $w$ on $\rn$
is said to belong to the \emph{class of Muckenhoupt weights}, $A_p(\rn)$, with $p\in(1,\,\fz)$, if it satisfies
\begin{align*}
[w]_{A_p(\rn)}:=\sup_{B\st\rn}\lf\{\frac{1}{|B|}\int_B w(x)\,dx\r\}\lf\{\frac{1}{|B|}\int_B[w(x)]^{1-p'}\,dx\r\}^{p-1}<\fz,
\end{align*}
where the supremum is taken over all balls $B$ of $\rn$.
It is well known that Muckenhoupt $A_p(\rn)$ weights can be characterized by
the weighted $L^p$ boundedness of the Hardy--Littlewood maximal function and the Hilbert transform
(see, for instance, \cite{hmw73,m72}).
In recent years, the sharp $A_p$ bounds  have been obtained for many operators,
for instance, the Hilbert transform \cite{p07}, the Riesz transform \cite{p08}
and the general Calder\'{o}n--Zygmund operator \cite{h12}.
In particular, Lerner \cite{Ler11} established the sharp
$A_p$ bounds for the so-called \emph{intrinsic Littlewood--Paley functions}
introduced by Wilson \cite{w07} (see also \cite[p.\,103]{w08}).
In what follows, for any $\beta\in(0,\,1]$, let $\mathcal{C}_\beta(\rn)$
be the family of all functions $\phi$,
defined on $\rn$, such that $\supp \phi \st \{x\in\rn \,: |x|\le 1\}$,
$\int_\rn \phi(x)\,dx=0$ and, for any $x_1,\,x_2\in\rn$,
$$|\phi(x_1)-\phi(x_2)|\le|x_1-x_2|^{\beta}.$$
For any $f\in L^1_{\loc}(\rn)$ (the set of all locally integrable functions)
and $(y,t)\in \mathbb{R}^{n+1}_{+}:=\rn \times (0,\fz)$,
let
\begin{align}\label{eq-4.0}
A_{\beta}(f)(y,\,t):=\sup_{\phi\in\mathcal{C}_\beta(\rn)}
|f\ast\phi_t(y)|=\sup_{\phi\in\mathcal{C}_\beta(\rn)}
\lf|\int_{\rn}\phi_{t}(y-z)f(z)\,dz\r|,
\end{align}
where, for any $t\in (0,\,\fz)$, $\phi_t(\cdot):=t^{-n}\phi({\cdot}/t)$.
For any $\beta\in (0,1]$ and $f\in L^1_{\loc}(\rn)$,
the \emph{intrinsic Littlewood--Paley $g$-function $g_{\beta}(f)$} and
the \emph{intrinsic Lusin area function $G_{\beta}(f)$} are
defined, respectively, by setting, for any $x\in\rn$,
$$g_{\beta}(f)(x):=\lf\{\int^\fz_0[A_{\beta}(f)(x,t)]^2\,\frac{dt}{t}\r\}^{1/2}$$
and
\begin{align}\label{eq-0.7}
G_{\beta,\,\az}(f)(x):=\lf\{\int^\fz_0\int_{\{y\in\rn :\ |y-x|<\az t\}}
[A_{\beta}(f)(y,\,t)]^2\,\dydt\r\}^{1/2},\ \forall\,\az\in(0,\,\fz).
\end{align}
In particular, if $\az=1$ in \eqref{eq-0.7},
we simply write $G_{\beta,\,\az}(f)$ as $G_{\beta}(f)$.
These square functions are independent of any particular kernel $\phi$
and pointwise comparable with each other.
Let $p\in(1,\,\fz)$.
Lerner \cite[Theorem 1.1]{Ler11} showed that there exists a positive constant $C$ such that,
for any $w\in A_p(\rn)$ and $f\in L^p(w)$,
\begin{align}\label{eq-0.0}
\lf\|g_\beta(f)\r\|_{L^p(w)}+\lf\|G_\beta(f)\r\|_{L^p(w)}\le C[w]_{A_p(\rn)}^{\max\{\frac12,\,\frac{1}{p-1}\}}\|f\|_{L^p(w)},
\end{align}
where the exponent $\max\{\frac12,\,\frac{1}{p-1}\}$ of $[w]_{A_p(\rn)}$ is the best possible
for any $p\in(1,\,\fz)$ (see \cite[p.\,487]{l06}).
Here and thereafter, for any $p\in(0,\,\fz)$ and any non-negative locally integrable function $w$ on $\rn$,
we use $L^p(w)$ to denote the space of all measurable functions $f$ on $\rn$ such that
$$\|f\|_{L^p(w)}:=\lf[\int_\rn |f(x)|^pw(x)\,dx\r]^{1/p}<\fz.$$

On the other hand,
let $n\geq 3$ and consider the Schr\"{o}dinger operator on the Euclidean space $\rn$,
\begin{align}\label{eq-op}
L:=-\Delta+V,
\end{align}
where $\Delta:=\sum_{j=1}^n\frac{\partial^2}{\partial x_j^2}$
denotes the \emph{Laplacian operator} on $\rn$ and $V$ is a
non-negative potential which belongs to the reverse H\"{o}lder class $RH_q(\rn)$
with $q\in(n/2,\,\fz)$.
Throughout this article, we always let $L$ be as in \eqref{eq-op}.
Recall that a non-negative measurable function $V$ on $\rn$ is said to belong to the
\emph{reverse H\"{o}lder class $RH_q(\rn)$}, $q\in[1,\,\fz]$, if there exists a positive
constant $C$ such that, for any ball $B\st\rn$,
\begin{align*}
\lf\{\frac{1}{|B|}\int_B[V(x)]^q\,dx\r\}^{1/q}\le C\frac{1}{|B|}\int_B V(x)\,dx
\end{align*}
with the usual modification made when $q=\fz$.
It is well known that, if $V\in RH_q(\rn)$ with $q\in(1,\,\fz)$, then $V(x)\,dx$ is a doubling
measure on $\rn$ (see, for instance, \cite[p.\,196]{s93}), namely, there exists a constant
$C_0\in(1,\,\fz)$ such that,
for any $x\in\rn$ and $r\in(0,\,\fz)$,
\begin{align}\label{eq-db}
\int_{B(x,\,2r)}V(y)\,dy\le C_0\int_{B(x,\,r)}V(y)\,dy.
\end{align}
For any $x\in\rn$, let
\begin{equation}\label{eq aux}
\rho(x):=\sup\lf\{r\in(0,\,\fz):\ \frac{1}{r^{n-2}}\int_{B(x,\,r)}V(y)\,dy\le 1\r\}.
\end{equation}
This auxiliary function was first introduced by Shen \cite[Definition 2.1]{sh},
where $V\in RH_q(\rn)$ with $q\in[n/2,\,\fz)$ and $n\geq 3$.
Since $V\in RH_q(\rn)$ implies that $V\in RH_{q+\varepsilon}(\rn)$
for some $\varepsilon\in(0,\,\fz)$ (see, for instance, \cite[p.\,219]{s93}),
it follows that the assumption $q\in(n/2,\,\fz)$ is equivalent to $q\in[n/2,\,\fz)$.
For convenience, we \emph{always assume} that $V\in RH_q(\rn)$ with $q\in(n/2,\,\fz)$.
Then, via $\rho$, Bongioanni et al. \cite{bhs11} introduced
a new class of Muckenhoupt weights adapted to the Schr\"{o}dinger operator $L:=-\Delta+V$ as follows.
In what follows, for any $x_B\in\rn$ and $r_B\in(0,\,\fz)$, let
$$B(x_B,\,r_B):=\{y\in\rn:\ |y-x_B|<r_B\}.$$
\begin{definition}[\cite{bhs11}]\label{def-1}
Let $p\in(1,\,\fz)$, $\tz\in[0,\,\fz)$, $n\geq 3$ and $\rho$ be as in \eqref{eq aux} with $V\in RH_q(\rn)$
and $q\in(n/2,\,\fz)$.
Then $A_p^{\rho,\,\tz}(\rn)$ is defined to be  the set of all
non-negative locally integrable functions $w$ on $\rn$ such that
\begin{align*}
[w]_{A_p^{\rho,\,\tz}(\rn)}
:=\sup_{B\st\rn}\frac{1}{\Psi_\tz(B)|B|}\int_B w(x)\,dx
\lf\{\frac{1}{\Psi_\tz(B)|B|}\int_B [w(x)]^{-\frac{1}{p-1}}\,dx\r\}^{p/p'}
<\fz,
\end{align*}
where $1/p+1/p'=1$, the supremum is taken over all balls $B:=B(x_B,\,r_B)$ of $\rn$ and
\begin{align}\label{eq-psi}
\Psi_\tz(B):=\lf[1+\frac{r_B}{\rho(x_B)}\r]^{\tz}.
\end{align}
\end{definition}

\begin{remark}\label{rem-3}
\begin{enumerate}
\item[(i)]
In particular, if $\tz=0$ and $V\equiv 0$ in \eqref{eq aux},
then $A_p^{\rho,\,\tz}(\rn)$ coincides with the
classical Muckenhoupt class $A_p(\rn)$.

\item[(ii)]
If $p$, $n$ and $\rho$ are as in Definition \ref{def-1}, and
$0\le \tz_1\le\tz_2<\fz$, then $A_p^{\rho,\,\tz_1}(\rn)\st A_p^{\rho,\,\tz_2}(\rn)$
and, for any $w\in A_p^{\rho,\,\tz_1}(\rn)$,
$$[w]_{A_p^{\rho,\,\tz_2}(\rn)}\le[w]_{A_p^{\rho,\,\tz_1}(\rn)}.$$
\end{enumerate}
\end{remark}

Bongioanni et al. \cite{bhs11} showed that $A_p^{\rho,\,\tz}(\rn)$
is a properly larger class than $A_p(\rn)$,
which contains the class $A_p(\rn)$ of classical Muckenhoupt weights,
and has most of the properties parallel to $A_p(\rn)$
(see also \cite[Proposition 2.3]{t14}).
Moreover, the weighted boundedness, related to $A_p^{\rho,\,\tz}(\rn)$,
of many operators associated to the Schr\"{o}dinger
operator $L:=-\Delta+V$ was obtained in \cite{bhs11,bch13,ly16,t11,t15,tz15}.
Recently, Li et al. \cite{lrw16} introduced the fractional weight class
$A_{p,\,q}^{\rho,\,\tz}(\rn)$ adapted to $L$ and
obtained the quantitative weighted boundedness of the fractional maximal function and
the fractional integral operator associated to $L$.

Motivated by \cite{bhs11,lrw16,Ler11}, in this article, we consider the \emph{Littlewood--Paley
functions} $g_L$, $S_L$ and $g_{L,\,\lambda}^\ast$
associated to the Schr\"{o}dinger operator $L$, which are defined, respectively, by setting,
for any $f\in L^2(\rn)$ and $x\in\rn$,
\begin{align}\label{eq-0.8}
g_L(f)(x):=\lf[\int_0^\fz \lf|tLe^{-tL}(f)(x)\r|^2\,\frac{dt}{t}\r]^{1/2},
\end{align}
\begin{align}\label{eq-0.5}
S_{L,\,\az}(f)(x):=\lf[\int_0^\fz\int_{|x-y|<\az t}\lf|t^2L e^{-t^2L}(f)(y)\r|^2\,\dydt\r]^{1/2},
\ \forall\,\az\in(0,\,\fz),
\end{align}
and
\begin{align*}
g_{L,\,\lambda}^\ast(f)(x):=\lf[\int_0^\fz\int_\rn\lf(\frac{t}{t+|x-y|}\r)^{\lambda n}
\lf|t^2L e^{-t^2L}(f)(y)\r|^2\,\dydt\r]^{1/2},\ \forall\,\lambda\in(0,\,\fz).
\end{align*}
In particular, if $\az=1$ in \eqref{eq-0.5}, we simply write $S_{L,\,\az}$ as $S_L$.
We obtain the following quantitative weighted boundedness of $g_L$, $S_L$ and $g_{L,\,\lambda}^\ast$.
\begin{theorem}\label{thm-1}
Let $n\geq 3$, $L$ be as in \eqref{eq-op} with $V\in RH_q(\rn)$ and $q\in(n/2,\,\fz)$,
and $\rho$ as in \eqref{eq aux}.
Assume that $p\in(1,\,\fz)$, $\tz\in[0,\,\fz)$ and $\lambda\in(3+\frac{2}{n}\max\{\frac{3\tz}{2}k_0,\,1\},\,\fz)$,
where $k_0:=\max\{\frac{\log_2{C_0}+2-n}{2-n/q},\,1\}$ and $C_0$ is as in \eqref{eq-db}.
Then there exists a positive constant $C$ such that,
for any $w\in A_p^{\rho,\,\tz}(\rn)$ and $f\in L^p(w)$,
\begin{align*}
\|g_L(f)\|_{L^p(w)}+\|S_L(f)\|_{L^p(w)}+\lf\|g_{L,\,\lambda}^\ast(f)\r\|_{L^p(w)}
\le C[w]_{A_p^{\rho,\,\tz}(\rn)}^{\max\{\frac12,\,\frac{1}{p-1}\}}\|f\|_{L^p(w)}.
\end{align*}
\end{theorem}

\begin{remark}\label{rem-4}
\begin{enumerate}
\item[(i)]
In particular, when $L:=-\Delta$ is the Laplacian operator in \eqref{eq-op},
then $g_{-\Delta}$, $S_{-\Delta}$ and $g_{-\Delta,\,\lz}^\ast$
are just the classical Littlewood--Paley functions [see \eqref{eq-g} and \eqref{eq-square} below]
and $A_p^{\rho,\,\tz}(\rn)=A_p(\rn)$.
In this case, since $g_{-\Delta}$ and $S_{-\Delta}$ are pointwise dominated by the intrinsic square function
$G_\beta$ with $\beta=1$ (see Lemma \ref{lem-9} below),
we find that the quantitative weighted boundedness of $g_{-\Delta}$ and $S_{-\Delta}$ in
Theorem \ref{thm-1} is implied by \eqref{eq-0.0}.

\item[(ii)]
We point out that the range $\lambda\in(3+\frac{2}{n}\max\{\frac{3\tz}{2}k_0,\,1\},\,\fz)$
in Theorem \ref{thm-1} may not be the best possible.
Let $\varphi\in C^\fz(\rn)$ satisfy $\supp\varphi\st\{x\in\rn:\ |x|\le1\}$
and $\int_\rn\varphi(x)\,dx=0$. For any $f\in L_{\loc}^1(\rn)$ and $x\in\rn$, define
\begin{align}\label{eq-4.4}
g_{\varphi,\,\lambda}^\ast(f)(x)
:=\lf[\int_0^\fz\int_\rn\lf(\frac{t}{t+|x-y|}\r)^{\lambda n}\lf|f\ast\varphi_t(y)\r|^2\,\dydt\r]^{1/2},
\ \forall\,\lambda\in(0,\,\fz),
\end{align}
where $\varphi_t(\cdot):=t^{-n}\varphi(\cdot/t)$ for any $t\in (0,\fz)$.
Lerner \cite[Theorem 1.2]{l06} showed that, if $p\in(1,\,\fz)$ and $\lambda\in(3,\,\fz)$,
then there exists a positive constant $C$ such that, for any $w\in A_p(\rn)$ and $f\in L^p(w)$,
\begin{align}\label{eq-4.2}
\lf\|M\lf(g_{\varphi,\,\lambda}^\ast(f)\r)\r\|_{L^p(w)}
\le C[w]_{A_p(\rn)}^{\max\{\frac12,\,\frac{1}{p-1}\}}\|M(f)\|_{L^p(w)},
\end{align}
where $M$ denotes the classical Hardy-Littlewood maximal operator as in \eqref{eq-0.4} below.
It was also pointed out in \cite[Theorem 1.2]{l06} that the
power $\max\{\frac12,\,\frac{1}{p-1}\}$ of $[w]_{A_p(\rn)}$ is sharp.
Moreover, by \eqref{eq-4.2} and the sharp boundedness of $M$ obtained by Buckley \cite{b93}
[see also \eqref{eq-0.3} below], it is easy to see that, if $p\in(1,\,\fz)$ and $\lz\in(3,\,\fz)$,
then there exists a positive constant $C$ such that, for any $w\in A_p(\rn)$ and $f\in L^p(w)$,
\begin{align}\label{eq-4.1}
\lf\|g_{\varphi,\,\lambda}^\ast(f)\r\|\le C[w]_{A_p(\rn)}^{\nu_p}\|f\|_{L^p(w)},
\end{align}
where $\nu_p:=\frac{1}{p-1}+\max\{\frac12,\,\frac{1}{p-1}\}$.
However, Lerner \cite{l06} conjectured that the best $\nu_p$ should be
$\max\{\frac12,\,\frac{1}{p-1}\}$.
Then it was proved in \cite[Corollary 1.3]{l08} that,
if $p\in(1,\,\fz)$ and $\lambda\in(2,\,\fz)$, then $\nu_p=\max\frac{1}{p-1}\{\frac{p}{2},\,1\}$
satisfies \eqref{eq-4.1}.
Observe that, for any $p\in(1,\,\fz)$,
$$\max\lf\{\frac12,\,\frac{1}{p-1}\r\}\le\max\frac{1}{p-1}\lf\{\frac{p}{2},\,1\r\}
<\frac{1}{p-1}+\max\lf\{\frac12,\,\frac{1}{p-1}\r\}.$$
Finally, we point out that Lerner \cite[Theorem 1.1]{Ler11} proved the conjecture that
\eqref{eq-4.1} holds true for $\nu_p=\max\{\frac12,\,\frac{1}{p-1}\}$.
Indeed, by \cite[Exercise 6.2]{w08},
we know that there exists a positive constant $C:=C_{(\beta,\,n)}$ such that,
for any $\beta\in(0,\,1]$, $\az\in[1,\,\fz)$, $f\in L^1_{\loc}(\rn)$ and $x\in\rn$,
$$G_{\beta,\,\az}(f)(x)\le C \az^{\frac{3n}{2}+\beta}G_\beta(f)(x),$$
where $G_{\beta,\,\az}$ is as in \eqref{eq-0.7}. Then, from this and
\eqref{eq-0.0}, we deduce that, if $p\in(1,\,\fz)$, $\beta\in(0,\,1]$
and $\lz\in(3+\frac{2\beta}{n},\,\fz)$,
then there exists a positive constant $C$ such that, for any $w\in A_p(\rn)$
and $f\in L^p(w)$,
\begin{align}\label{eq-4.3}
\lf\|g_{\beta,\,\lambda}^\ast(f)\r\|_{L^p(w)}\le C
[w]_{A_p(\rn)}^{\max\{\frac12,\,\frac{1}{p-1}\}}\|f\|_{L^p(w)},
\end{align}
where $g_{\beta,\,\lambda}^\ast(f)$ is the intrinsic $g_\lz^\ast$-function
which is defined as in \eqref{eq-4.4} with $f\ast\varphi_t(y)$ therein
replaced by $A_\beta(f)(y,\,t)$ of \eqref{eq-4.0}.
In particular, when $L:=-\Delta$ is the Laplacian operator in \eqref{eq-op},
we could choose $\tz:=0$ in Theorem \ref{thm-1}. Then, by Theorem \ref{thm-1},
we conclude that, if $\lz\in(3+\frac{2}{n},\,\fz)$ and $p\in(1,\,\fz)$, then
there exists a positive constant $C$ such that, for any $w\in A_p(\rn)$ and $f\in L^p(w)$,
$$\lf\|g_{-\Delta,\,\lambda}^\ast(f)\r\|_{L^p(w)}
\le C[w]_{A_p(\rn)}^{\max\{\frac12,\,\frac{1}{p-1}\}}\|f\|_{L^p(w)}.$$
The range $\lz\in(3+\frac{2}{n},\,\fz)$ in this case coincides with the range
$\lz\in(3+\frac{2\beta}{n},\,\fz)$,
because, by Lemma \ref{lem-9} below, we know that there exists a positive
constant $C$ such that, for any $f\in L^p(w)$ and $x\in\rn$,
$$g_{-\Delta,\,\lambda}^\ast(f)(x)\le Cg_{1,\,\lambda}^\ast(f)(x).$$
However, we do not know whether or not Theorem \ref{thm-1} holds true
for smaller $\lz$, for instance, $\lz\in(2,\,3+\frac{2}{n}\max\{\frac{3\tz}{2}k_0,\,1\}]$,
where $k_0$ is as in Theorem \ref{thm-1}.
\end{enumerate}
\end{remark}
We prove Theorem \ref{thm-1} in Section \ref{s2}
by borrowing some ideas from \cite{bhs11}, where Bongioanni et al. introduced a $\rho$-localized method
to obtain the boundedness of operators associated to the Schr\"{o}dinger operator $L$
on weighted Lebesgue spaces. This is quite different from the method used in \cite{lrw16}.
Besides this, we also need a quantitative version of the
extrapolation theorem for $A_{p}^{\rho,\,\tz}(\rn)$ weights (see Lemma \ref{lem-4} below).
The key idea is to compare
the operators $g_L$ and $S_L$, respectively,
with the classical ones $g_{-\Delta}$ and $S_{-\Delta}$
when they are restricted to local regions related to $\rho$ (see Section \ref{s-2} below).
This leads us to establish the quantitative weighted estimates for $\rho$-localized
classical maximal functions and Littlewood--Paley functions
(see Lemma \ref{lem-2} below).
To this end, we make use of some results related to the extensions of weights and the
sharp weighted estimates of the intrinsic Lusin area function $G_{\beta}$ [see \eqref{eq-0.0} above].

For the quantitative weighted boundedness of $g_{L,\,\lz}^\ast$, we also have the following result.
\begin{theorem}\label{thm-2}
Let $n\geq 3$, $L$ be as in \eqref{eq-op} with $V\in RH_q(\rn)$ and $q\in(n/2,\,\fz)$,
and $\rho$ as in \eqref{eq aux}.
Assume that $p\in(1,\,\fz)$, $\tz\in[0,\,\fz)$ and $\lz\in(2[1+\frac{2\tz}{n}],\,\fz)$.
Then there exists a positive constant $C$ such that,
for any $w\in A_p^{\rho,\,\tz}(\rn)$ and $f\in L^p(w)$,
\begin{align*}
\lf\|g_{L,\,\lambda}^\ast(f)\r\|_{L^p(w)}
\le C[w]_{A_p^{\rho,\,\tz}(\rn)}^{\max\{1,\,\frac{1}{p-1}\}+\max\{\frac12,\,\frac{1}{p-1}\}}\|f\|_{L^p(w)}.
\end{align*}
\end{theorem}
Compared with Theorem \ref{thm-1}, the range of $\lambda$ in Theorem \ref{thm-2}
is wider than the corresponding one in Theorem \ref{thm-1}; however, the power index
of the weight constant $[w]_{A_p^{\rho,\,\tz}(\rn)}$ in Theorem \ref{thm-2} is larger
than the corresponding one in Theorem \ref{thm-1}.

We prove Theorem \ref{thm-2} in Section \ref{s2}
by borrowing some ideas from the proof of \cite[Proposition 3.2]{mp17},
where Martell et al. established some sharp weighted $L^p$ estimates of square functions with various angles
by using the extrapolation theorem for classical Muckenhoupt $A_p$ weights.

We end this section by making some conventions on notation.
In this article,
we denote by $C$ a positive constant which is independent of the main parameters,
but it may vary from line to line.
We also use $C_{(\az, \bz,\ldots)}$ to denote a positive constant depending on
the parameters $\az$, $\bz$, $\ldots$.
The \emph{symbol $f\ls g$} means that $f\le Cg$.
If $f\ls g$ and $g\ls f$,  we then write $f\sim g$. We also use the following
convention: If $f\le Cg$ and $g=h$ or $g\le h$, we then write $f\ls g\sim h$
or $f\ls g\ls h$, \emph{rather than} $f\ls g=h$
or $f\ls g\le h$. Let $\nn:=\{1,\,2,\,\ldots\}$ and $\zz_+:=\nn\cup\{0\}$.
For any measurable subset $E$ of $\rn$, we denote by $E^\com$ the
\emph{set $\rn\setminus E$} and also by ${\mathbf 1}_E$ its \emph{characteristic function}.
We also denote by $C_c^\fz(\rn)$ the set of all infinitely differential functions with compact supports.
For any ball
$B:=B(x_B,r_B):=\{y\in\rn:\ |x_B-y|<r_B\}\st\rn$,
$x_B\in\rn$ and $r_B\in(0,\,\fz)$, and $\az\in(0,\,\fz)$,
we let $\az B:=B(x_B,\az r_B)$.
For any $p\in[1,\,\fz]$, $p'$ denotes its \emph{conjugate number},
namely, $1/p+1/p'=1$.

\section{Preliminaries and extrapolation theorem}\label{s2}

In this section, we first recall some notions and properties
related to the Schr\"{o}dinger operator $L$.
Then we establish a quantitative version of the extrapolation theorem
for $A_{p}^{\rho,\,\tz}(\rn)$ weights.

For the auxiliary function $\rho(\cdot)$ in \eqref{eq aux},
we have the following lemma,
which is just \cite[Lemma 1.4]{sh}.

\begin{lemma}[\cite{sh}]\label{lem aux}
Let $n\geq 3$, $\rho$ be as in \eqref{eq aux} with $V\in RH_q(\rn)$ and $q\in(n/2,\,\fz)$.
Then there exist
$C\in(0,\,\fz)$ and $N_0\in(0,\,\fz)$ such that, for any $x$, $y\in\rn$,
$$C^{-1}\rho(x)\lf[1+\frac{|x-y|}{\rho(x)}\r]^{-N_0}\le\rho(y)
\le C\rho(x)\lf[1+\frac{|x-y|}{\rho(x)}\r]^{\frac{N_0}{N_0+1}}.$$
\end{lemma}

\begin{remark}\label{rem-7}
Let $C_0$ be as in \eqref{eq-db}. By carefully checking the proof of \cite[Lemma 1.4]{sh},
we find that $N_0$ in Lemma \ref{lem aux} could be any constant such that
$N_0\in(\max\{\frac{\log_2{[C_02^{(n/q)-n}]}}{2-(n/q)},\,0\},\,\fz)$.
\end{remark}

The following estimate for the kernel of $e^{-tL}$ can be found in \cite[Theorem 1]{Ku00}
(see also \cite[Proposition 2.3]{ly16}).
\begin{lemma}[\cite{Ku00}]\label{lem-8}
Let $L$ be as in \eqref{eq-op}, $n\geq 3$ and $\rho$ be as in \eqref{eq aux} with $V\in RH_q(\rn)$
and $q\in(n/2,\,\fz)$.
Assume that $K_t(\cdot,\,\cdot)$ is the integral kernel
of the heat semigroup $\{e^{-tL}\}_{t\geq 0}$ generated by $L$.
Then, for any given $k\in\zz_+$ and $N\in(0,\,\fz)$,
there exist positive constants $C_{(k,\,N)}$ and $c_{(k)}$ such that,
for any $t\in(0,\,\fz)$ and $(x,y)\in\rn\times\rn$,
\begin{equation*}
\lf|\frac{\partial^k}{\partial t^k}K_t(x,\,y)\r|\le\frac{C_{(k,\,N)}}{t^{n/2+k}}e^{-c_{(k)}\frac{|x-y|^2}{t}}
\lf[1+\frac{\sqrt{t}}{\rho(x)}+\frac{\sqrt{t}}{\rho(y)}\r]^{-N}.
\end{equation*}
\end{lemma}

Let $\tz\in[0,\,\fz)$.
The \emph{Hardy--Littlewood type maximal operator $M_\tz$}, associated to $L$,
is defined by setting, for any $f\in L_{\loc}^1(\rn)$ and $x\in\rn$,
\begin{align}\label{eq-fm}
M_\tz(f)(x):=\sup_{B\ni x}\frac{1}{\Psi_\tz(B)|B|}\int_B|f(y)|\,dy,
\end{align}
where the supremum is taken over all balls $B$ of $\rn$
containing $x$ and $\Psi_\tz(B)$ is as in \eqref{eq-psi}.

The following lemma is a particular case of \cite[Theorem 1.3]{lrw16},
which establishes the quantitative estimate for $M_\tz$.
\begin{lemma}[\cite{lrw16}]\label{lem-10}
Let $n\geq 3$, $\rho$ be as in \eqref{eq aux} with $V\in RH_q(\rn)$ and $q\in(n/2,\,\fz)$,
$\tz\in[0,\,\fz)$, $p\in(1,\,\fz)$ and $1/p+1/p'=1$.
Then there exists a positive constant $C$ such that, for any
$w\in A_{p}^{\rho,\,\eta}(\rn)$, with $\eta:=\tz/p'$, and $f\in L^p(w)$,
\begin{align*}
\lf\|M_\tz(f)\r\|_{L^p(w)}\le C[w]_{A_p^{\rho,\,\eta}(\rn)}^{\frac{1}{p-1}}\|f\|_{L^p(w)}.
\end{align*}
\end{lemma}

\begin{remark}\label{rem-6}
Recall that the classical \emph{Hardy-Littlewood maximal operator} $M$ is defined
by setting, for any $f\in L^1_{\loc}(\rn)$ and $x\in\rn$,
\begin{align}\label{eq-0.4}
M(f)(x):=\sup_{B\st\rn}\frac{1}{|B|}\int_B|f(y)|\,dy,
\end{align}
where the supremum is taken over all balls $B$ of $\rn$ containing $x$.
If $\tz=0$ and $p\in(1,\,\fz)$, then, in this case,
the conclusion of Lemma \ref{lem-10} just becomes that
there exists a positive constant $C$ such that, for any $w\in A_p(\rn)$ and $f\in L^p(w)$,
\begin{align}\label{eq-0.3}
\|M(f)\|_{L^p(w)}\le C[w]_{A_p(\rn)}^{\frac{1}{p-1}}\|f\|_{L^p(w)}.
\end{align}
Observe that \eqref{eq-0.3} is just the sharp weighted bound for $M$ obtained by Buckley \cite{b93}.
\end{remark}

We have the following quantitative version of the extrapolation theorem for $A_p^{\rho,\,\tz}(\rn)$ weights.
\begin{lemma}\label{lem-4}
Let $n\geq 3$, $\rho$ be as in \eqref{eq aux} with $V\in RH_q(\rn)$ and $q\in(n/2,\,\fz)$,
$p_0\in(1,\,\fz)$, $\gamma\in[0,\,\fz)$ and $T$ be an operator defined on $C_c^\fz(\rn)$.
Suppose that there exist positive constants $c$ and $\eta$ such that,
for any $w\in A_{p_0}^{\rho,\,\gamma}(\rn)$ and $f\in L^{p_0}(w)$,
\begin{align*}
\lf\|T(f)\r\|_{L^{p_0}(w)}\le c[w]^\eta_{A_{p_0}^{\rho,\,\gamma}(\rn)}\|f\|_{L^{p_0}(w)}.
\end{align*}
Then, for any $p\in(1,\,\fz)$, there exists a positive constant $C$ such that,
for any $w\in A_p^{\rho,\,\theta}(\rn)$,
$\tz:=\gamma/p_0'$ if $p>p_0$ or $\tz:=\gamma/p_0$ if $p<p_0$, and $f\in L^p(w)$,
\begin{align*}
\lf\|T(f)\r\|_{L^p(w)}\le C[w]^{\eta\max\{1,\,\frac{p_0-1}{p-1}\}}_{A_p^{\rho,\,\tz}(\rn)}\|f\|_{L^p(w)}.
\end{align*}
\end{lemma}

\begin{remark}\label{rem-1}
\begin{enumerate}
\item[(i)]
We point out that Tang \cite[Theorems 3.1 and 3.2]{t14} also
obtained an extrapolation theorem for $A_p^{\rho,\,\tz}(\rn)$ weights
on weighted Lebesgue and Lorentz spaces.
Moreover, Bongioanni et al. \cite[Theorem 1]{bch13}
obtained an extrapolation theorem for a class of abstract weights,
which includes $A_p^{\rho,\,\tz}(\rn)$ weights as a special case,
on weighted Lebesgue spaces. However, neither of the above two extrapolation theorems
is quantitative.

\item[(ii)]
Lemma \ref{lem-4} can be written in terms of pairs of functions as follows.
Let $\mathcal{F}$ be a given family of pairs $(f,\,g)$
of non-negative measurable functions on $\rn$ and
$\rho$ as in \eqref{eq aux} with $V\in RH_q(\rn)$ and $q\in(n/2,\,\fz)$.
Suppose that there exist positive constants $c$ and $\eta$ such that,
for some fixed $p_0\in[1,\,\fz)$ and for any $w\in A_{p_0}^{\rho,\,\gamma}(\rn)$ with $\gamma\in[0,\,\fz)$,
\begin{align*}
\int_\rn [f(x)]^{p_0}w(x)\,dx\le c[w]^{p_0\eta}_{A_{p_0}^{\rho,\,\gamma}(\rn)}\int_\rn [g(x)]^{p_0}w(x)\,dx,
\quad\forall\ (f,\,g)\in\mathcal{F}.
\end{align*}
Then there exists a positive constant $C$ such that,
for any $p\in(1,\,\fz)$ and $w\in A_p^{\rho,\,\tz}(\rn)$
with $\tz:=\gamma/p_0'$ if $p>p_0$ or $\tz:=\gamma/p_0$ if $p<p_0$,
\begin{align*}
\int_\rn [f(x)]^{p}w(x)\,dx\le C[w]^{p\eta\max\{1,\,\frac{p_0-1}{p-1}\}}_{A_p^{\rho,\,\tz}(\rn)}
\int_\rn [g(x)]^{p}w(x)\,dx,
\quad\forall\ (f,\,g)\in\mathcal{F}.
\end{align*}
As usual, in the above two inequalities, we always assume that the left-hand side is finite.
\end{enumerate}
\end{remark}

The key to prove Lemma \ref{lem-4} is a version of the \emph{Rubio de Francia iteration algorithm}
for $A_p^{\rho,\,\tz}(\rn)$ weights (see Lemma \ref{lem-5} below).
Once it is established, the proof of Lemma \ref{lem-4}
is completely similar to that of \cite[Theorem 2.1]{lmpt10} (see also \cite[Theorem 1]{dgpp05})
and we omit the details.

\begin{lemma}\label{lem-5}
Let $n\geq 3$, $\rho$ be as in \eqref{eq aux} with $V\in RH_q(\rn)$ and $q\in(n/2,\,\fz)$,
$1\le r_0<r<\fz$, $\tz\in[0,\,\fz)$,
$w\in A_r^{\rho,\,\tz}(\rn)$ and $g\in L^{\frac{r}{r-r_0}}(w)$ be non-negative.
Then there exists a function $G$ such that
\begin{enumerate}
\item[i)] $g\le G$;
\item[ii)] $\|G\|_{L^{\frac{r}{r-r_0}}(w)}\le 2\|g\|_{L^{\frac{r}{r-r_0}}(w)}$;
\item[iii)] $Gw\in A_{r_0}^{\rho,\,\gamma}(\rn)$ with
$[Gw]_{A_{r_0}^{\rho,\,\gamma}(\rn)}\le c[w]_{A_r^{\rho,\,\tz}(\rn)}$,
\end{enumerate}
where $\gamma:=r\tz$ and $c$ is a positive constant independent of $w$.
\end{lemma}

\begin{proof}
Let $s:=\frac{r-r_0}{r-1}$.
By the fact that $1\le r_0<r<\fz$, we find that $s\in(0,\,1]$.
For any $w$ and $g$ as in this lemma, let
\begin{align*}
R(g):=\lf[M_{\gamma}\lf(g^{1/s}w\r)w^{-1}\r]^s,
\end{align*}
where $\gamma\in(0,\,\fz)$ is fixed later and $M_\gamma$ is as in \eqref{eq-fm}
with $\tz=\gamma$ therein.
Then we have
\begin{align}\label{eq-1.3}
\int_\rn [R(g)(x)]^{\frac{r}{r-r_0}}w(x)\,dx
&=\int_\rn\lf[M_{\gamma}\lf(g^{1/s}w\r)(x)w^{-1}(x)\r]^{\frac{r}{r-1}}w(x)\,dx\\
&=\int_\rn\lf[M_{\gamma}\lf(g^{1/s}w\r)(x)\r]^{r'}[w(x)]^{-\frac{1}{r-1}}\,dx.\noz
\end{align}
By the fact that $w\in A_r^{\rho,\,\tz}(\rn)$ and Definition \ref{def-1},
we know that $w^{-\frac{1}{r-1}}\in A_{r'}^{\rho,\,\tz}(\rn)$ and
\begin{equation}\label{eq-1.4}
[w]_{A_r^{\rho,\,\tz}(\rn)}=\lf[w^{-\frac1{r-1}}\r]_{A_{r'}^{\rho,\,\tz}(\rn)}^{\frac{r}{r'}}.
\end{equation}
From this and Lemma \ref{lem-10}, we deduce that
\begin{align*}
\int_\rn\lf[M_{\gamma}\lf(g^{1/s}w\r)(x)\r]^{r'}[w(x)]^{-\frac{1}{r-1}}\,dx
&\ls[w^{-\frac{1}{r-1}}]_{A_{r'}^{\rho,\,\tz}(\rn)}^\frac{r'}{r'-1}
\int_\rn \lf\{[g(x)]^{1/s}w(x)\r\}^{r'}[w(x)]^{-\frac{1}{r-1}}\,dx\\
&\sim[w^{-\frac{1}{r-1}}]_{A_{r'}^{\rho,\,\tz}(\rn)}^r\int_\rn [g(x)]^{\frac{r}{r-r_0}}w(x)\,dx,
\end{align*}
where we fix $\gamma=r\tz$.
This, combined with \eqref{eq-1.3} and \eqref{eq-1.4}, implies that
\begin{align}\label{eq-1.5}
\|R(g)\|_{L^{\frac{r}{r-r_0}}(w)}
\ls[w]_{A_{r}^{\rho,\,\tz}(\rn)}^s\|g\|_{L^{\frac{r}{r-r_0}}(w)}.
\end{align}
Denote the operator norm $\|R\|_{L^{\frac{r}{r-r_0}}(w)\to L^{\frac{r}{r-r_0}}(w)}$ simply by $\|R\|$.
Then, by \eqref{eq-1.5}, we find that $\|R\|\ls[w]^s_{A^{\rho,\,\tz}_r(\rn)}$.
For any $w$ and $g$ as in this lemma, define
$$G:=\sum_{k=0}^\fz\frac{R^k(g)}{2^k\|R\|^k},$$
where $R^0:=I$ is the identity operator. It is easy to see that
$g\le G$ and
\begin{align*}
\|G\|_{L^{\frac{r}{r-r_0}}(w)}\le\sum_{k=0}^\fz\frac{\|R\|^k\|g\|_{L^{\frac{r}{r-r_0}}(w)}}{2^k\|R\|^k}
=2\|g\|_{L^{\frac{r}{r-r_0}}(w)}.
\end{align*}
Hence, (i) and (ii) hold true.

Next, we prove (iii), namely, for any $w\in A_r^{\rho,\,\tz}(\rn)$ and $G$ as above,
$[Gw]_{A_{r_0}^{\rho,\,\gamma}(\rn)}\ls[w]_{A_r^{\rho,\,\tz}(\rn)}$.
Indeed, for any $w\in A_r^{\rho,\,\tz}(\rn)$ and $G$ as above, we have
\begin{align}\label{eq-1.6x}
[Gw]_{A_{r_0}^{\rho,\,\gamma}(\rn)}
=\sup_{B\st\rn}\lf[\frac{1}{\Psi_\gamma(B)|B|}\int_B G(y)w(y)\,dy\r]
\lf\{\frac{1}{\Psi_\gamma(B)|B|}\int_B[G(y)w(y)]^{-\frac{1}{r_0-1}}\,dy\r\}^{r_0-1}.
\end{align}
For the first factor of the right-hand side of \eqref{eq-1.6x},
applying the H\"{o}lder inequality with exponents $1/s$ and $(1/s)'=\frac{1}{1-s}$,
we obtain
\begin{align}\label{eq-1.6}
\frac{1}{\Psi_\gamma(B)|B|}\int_B G(y)w(y)\,dy
\le\lf\{\frac{1}{\Psi_\gamma(B)|B|}[G(y)]^{1/s}w(y)\,dy\r\}^s
\lf\{\frac{1}{\Psi_\gamma(B)|B|}w(y)\,dy\r\}^{1-s}.
\end{align}

Next, we estimate the second factor of the right-hand side of \eqref{eq-1.6x}.
By the above definition of $G$, we have
\begin{align*}
R(G)=\sum_{k=0}^\fz\frac{R^{k+1}(g)}{2^k\|R\|^k}
=2\|R\|\sum_{k=0}^\fz\frac{R^{k+1}(g)}{2^{k+1}\|R\|^{k+1}}=2\|R\|(G-g)\le 2\|R\|G,
\end{align*}
namely,
$[M_{\gamma}(G^{1/s}w)w^{-1}]^s\le 2\|R\|G$.
This, together with the fact that $\|R\|\ls[w]^s_{A^{\rho,\,\tz}_r(\rn)}$, implies that
$M_{\gamma}(G^{1/s}w)\ls[w]_{A_r^{\rho,\,\tz}(\rn)}G^{1/s}w$.
Thus, for any $x\in\rn$ and any ball $B$ containing $x$, it holds true that
\begin{align*}
\frac{1}{\Psi_\gamma(B)|B|}\int_B [G(y)]^{1/s}w(y)\,dy\ls [w]_{A_r^{\rho,\,\tz}(\rn)}[G(x)]^{1/s}w(x),
\end{align*}
namely,
\begin{align*}
[w(x)]^{-s}[w]_{A_r^{\rho,\,\tz}(\rn)}^{-s}\lf\{\frac{1}{\Psi_\gamma(B)|B|}\int_{B}[G(y)]^{1/s}w(y)\,dy\r\}^s
\ls G(x).
\end{align*}
From this and the fact that $s=\frac{r-r_0}{r-1}$, we deduce that
\begin{align*}
&\lf\{\frac{1}{\Psi_\gamma(B)|B|}\int_B[G(y)w(y)]^{-\frac{1}{r_0-1}}\,dy\r\}^{r_0-1}\\
&\hs\ls\lf\{\frac{1}{\Psi_\gamma(B)|B|}\int_B[w(y)]^{-\frac{1-s}{r_0-1}}
[w]_{A_r^{\rho,\,\tz}(\rn)}^{\frac{s}{r_0-1}}
\lf[\frac{1}{\Psi_\gamma(B)|B|}\int_{B}\{G(z)\}^{1/s}w(z)\,dz\r]^{-\frac{s}{r_0-1}}\,dy\r\}^{r_0-1}\\
&\hs\ls[w]_{A_r^{\rho,\,\tz}(\rn)}^{s}\lf\{\frac{1}{\Psi_\gamma(B)|B|}\int_{B}[G(z)]^{1/s}w(z)\,dz\r\}^{-s}
\lf\{\frac{1}{\Psi_\gamma(B)|B|}\int_B[w(y)]^{-\frac{1}{r-1}}\,dy\r\}^{r_0-1}.
\end{align*}
By this, \eqref{eq-1.6}, \eqref{eq-1.6x} and the fact that
$[w]_{A_r^{\rho,\,\gamma}(\rn)}\le[w]_{A_r^{\rho,\,\tz}(\rn)}$ [see Remark \ref{rem-3}(ii)], we conclude that
\begin{align*}
[Gw]_{A_{r_0}^{\rho,\,\gamma}(\rn)}
&\ls\sup_{B\st\rn}[w]_{A_r^{\rho,\,\tz}(\rn)}^{s}
\lf[\frac{1}{\Psi_\gamma(B)|B|}\int_B w(y)\,dy\r]^{1-s}
\lf\{\frac{1}{\Psi_\gamma(B)|B|}\int_B[w(y)]^{-\frac{1}{r-1}}\,dy\r\}^{r_0-1}\\
&\sim [w]_{A_r^{\rho,\,\tz}(\rn)}^{s}[w]_{A_r^{\rho,\,\gamma}(\rn)}^{1-s}
\ls [w]_{A_r^{\rho,\,\tz}(\rn)}.
\end{align*}
This shows (iii) and hence finishes the proof of Lemma \ref{lem-5}.
\end{proof}

\section{Localized weights and operators}\label{s3}

In this section, we recall the $\rho$-localized operators and weights
introduced by Bongioanni et al. \cite{bhs11}.
Then we establish quantitative weighted estimates of
$\rho$-localized maximal functions and Littlewood--Paley operators.

We begin with recalling the \emph{radial maximal function} $R_{-\Delta}$
and the \emph{non-tangential maximal function} $M^\ast_{-\Delta}$,
associated to the Laplacian operator $-\Delta$,
which are defined, respectively, by setting, for any $f\in C_c^\fz(\rn)$ and $x\in\rn$,
\begin{align*}
R_{-\Delta}(f)(x)
:=\sup_{t\in(0,\,\fz)}\lf|e^{-t\Delta}(f)(x)\r|
:=\sup_{t\in(0,\,\fz)}\lf|\frac{1}{(4\pi t)^{n/2}}\int_\rn e^{-\frac{|y-z|^2}{4t}}f(z)\,dz\r|
\end{align*}
and
\begin{align*}
M^\ast_{-\Delta,\,\az}(f)(x)
:=\sup_{(y,\,t)\in\Gamma_\az(x)}\lf|e^{-t^2\Delta}(f)(y)\r|
:=\sup_{(y,\,t)\in\Gamma_\az(x)}\lf|\frac{1}{(4\pi t^2)^{n/2}}\int_\rn e^{-\frac{|y-z|^2}{4t^2}}f(z)\,dz\r|,
\end{align*}
where $\az\in(0,\,\fz)$ and $\bgz_\az(x):=\{(y,t)\in\rr^{n+1}_+:\ |x-y|<\az t\}$
denotes the \emph{cone of aperture $\az$ with  vertex $x$}.

It is easy to see that there exists a positive constant $C$ such that,
for any $\az\in[1,\,\fz)$, $f\in C_c^\fz(\rn)$ and $x\in\rn$,
\begin{align}\label{eq-0.2}
R_{-\Delta}(f)(x)
&\le M^\ast_{-\Delta,\,\az}(f)(x)\\
&=\sup_{(y,\,t)\in\Gamma_\az(x)}\lf|e^{-t^2\Delta}(f)(y)\r|
\ls\sup_{(y,\,t)\in\Gamma_\az(x)}\int_\rn\frac{1}{t^n}e^{-\frac{|y-z|^2}{4t^2}}|f(z)|\,dz\noz\\
&\sim\sup_{(y,\,t)\in\Gamma_\az(x)}\lf[\frac{1}{t^n}\int_{B(y,\,\az t)}|f(z)|\,dz+
\sum_{j=0}^\fz e^{-\frac{(2^j\az)^2}{4}}\frac{1}{t^n}\int_{B(y,\,2^{j+1}\az t)\setminus B(y,\,2^{j}\az t)}|f(z)|\,dz\r]\noz\\
&\ls\sup_{t\in(0,\,\fz)}
\lf[\az^n \frac{1}{(2\az t)^n}\int_{B(x,\,2\az t)}|f(z)|\,dz\r.\noz\\
&\hs\lf.+\sum_{j=1}^\fz e^{-\frac{(2^j\az)^2}{4}}(2^j\az)^n\frac{1}{(2^{j+2}\az t)^n}\int_{B(x,\,2^{j+2}\az t)}|f(z)|\,dz\r]\noz\\
&\le C\az^n M(f)(x),\noz
\end{align}
where $M$ is the Hardy--Littlewood maximal function as in \eqref{eq-0.4}.

In particular, if $L:=-\Delta$ in \eqref{eq-0.8} and\eqref{eq-0.5},
then $g_{-\Delta}$ and $S_{-\Delta,\,\az}$ are just the classical
Littlewood--Paley functions. Indeed, we have, for any $f\in C_c^\fz(\rn)$ and $x\in\rn$,
\begin{align}\label{eq-g}
g_{-\Delta}(f)(x)
&=\lf[\int_0^\fz\lf|t\Delta e^{-t\Delta}(f)(x)\r|^2\,\frac{dt}{t}\r]^{1/2}
=\lf[\int_0^\fz\lf|\psi_{\sqrt{t}}\ast f(x)\r|^2\,\frac{dt}{t}\r]^{1/2}\\
&=\sqrt{2}\lf[\int_0^\fz\lf|\psi_{t}\ast f(x)\r|^2\,\frac{dt}{t}\r]^{1/2}\noz
\end{align}
and
\begin{align}\label{eq-square}
S_{-\Delta,\,\az}(f)(x)
&=\lf[\int_0^\fz\int_{|y-x|<\az t}\lf|t^2\Delta e^{-t^2\Delta}(f)(y)\r|^2\dydt\r]^{1/2}\\
&=\lf[\int_0^\fz\int_{|y-x|<\az t}\lf|\psi_t\ast f(y)\r|^2\dydt\r]^{1/2},\noz
\end{align}
where $\psi(\cdot):=\frac12 |\cdot|^2e^{-|\cdot|^2/4}-ne^{-|\cdot|^2/4}$ and $\psi_t(\cdot):=t^{-n}\psi(\cdot/t)$.

For $g_{-\Delta}$ and $S_{-\Delta,\,\az}$, we have the following estimates.
\begin{lemma}\label{lem-9}
Let $\az\in(0,\,\fz)$.
Then there exists a positive constant $C:=C_{(n)}$ such that,
for any $f\in C_c^\fz(\rn)$ and $x\in\rn$,
\begin{align}\label{eq-1.14}
S_{-\Delta,\,\az}(f)(x)\le C \az^{\frac{3n}{2}+1}G_{1}(f)(x)
\end{align}
and
\begin{align}\label{eq-1.16}
g_{-\Delta}(f)(x)\le CG_{1}(f)(x),
\end{align}
where $G_1$ is as in \eqref{eq-0.7} with $\az=\beta=1$ therein.
\end{lemma}

\begin{proof}
For any $\beta\in(0,\,1]$ and $\epsilon\in(0,\,\fz)$,
let $\mathcal{C}_{(\beta,\,\epsilon)}(\rn)$ be the set of all functions $\phi$, defined on $\rn$,
such that, for any $x,\,\tilde{x}\in\rn$, $|\phi(x)|\le(1+|x|)^{-n-\epsilon}$,
$$|\phi(x)-\phi(\tilde{x})|\le|x-\tilde{x}|^\beta\lf[(1+|x|)^{-n-\epsilon}+(1+|\tilde{x}|)^{-n-\epsilon}\r]$$
and $\int_\rn\phi(x)\,dx=0$.
For any $\beta\in(0,\,1]$, $\epsilon\in(0,\,\fz)$, $\az\in(0,\,\fz)$ and $f\in C_c^\fz(\rn)$,
define
\begin{align*}
\wz{g}_{(\beta,\,\epsilon)}(f)(x)
:=\lf\{\int_0^\fz\lf[\wz{A}_{(\beta,\,\epsilon)}(f)(x,\,t)\r]^2\,\frac{dt}{t}\r\}^{1/2}
\end{align*}
and
\begin{align*}
\wz{G}_{(\beta,\,\epsilon),\,\az}(f)(x)
:=\lf\{\int_0^\fz\int_{|x-y|<\az t}\lf[\wz{A}_{(\beta,\,\epsilon)}(f)(y,\,t)\r]^2\,\dydt\r\}^{1/2},
\end{align*}
where
\begin{align*}
\wz{A}_{(\beta,\,\epsilon)}(f)(y,\,t):=\sup_{\phi\in\mathcal{C}_{(\beta,\,\epsilon)}(\rn)}
\lf|\phi_t\ast f(y)\r|.
\end{align*}
In particular, if $\az=1$, we simply write $\wz{G}_{(\beta,\,\epsilon),\,\az}$ as $\wz{G}_{(\beta,\,\epsilon)}$.
These square functions are introduced by Wilson \cite[p.\,775]{w07} (see also \cite[p.\,117]{w08}).

Let $\psi$ be as in \eqref{eq-square}.
It is easy to see that $\int_\rn\psi(x)\,dx=0$.
Moreover, we know that, for any $\epsilon\in(0,\,\fz)$ and $x\in\rn$,
$|\psi(x)|\ls(1+|x|)^{-n-\epsilon}$.
For any $x,\,\tilde{x}\in\rn$, if $|x-\tilde{x}|>1$, then we have
$$|\psi(x)-\psi(\tilde{x})|\ls (1+|x|)^{-n-\epsilon}+(1+|\tilde{x}|)^{-n-\epsilon}
\ls|x-\tilde{x}|[(1+|x|)^{-n-\epsilon}+(1+|\tilde{x}|)^{-n-\epsilon}].$$
If $|x-\tilde{x}|\le 1$, we consider two cases. When $|x|\le2$, we have $|\tilde{x}|\le|x|+|x-\tilde{x}|<3$
and hence $$|\psi(x)-\psi(\tilde{x})|\ls|x-\tilde{x}|\ls|x-\tilde{x}|[(1+|x|)^{-n-\epsilon}+(1+|\tilde{x}|)^{-n-\epsilon}].$$
When $|x|>2$, we find that, for any $\tz\in[0,\,1]$,
$$\frac{|x|}{2}<|x|-|x-\tilde{x}|\le|x+\tz(\tilde{x}-x)|\le |x|+|x-\tilde{x}|<\frac{3}{2}|x|.$$
From this and the mean value theorem, we deduce that
there exists some $\tz\in[0,\,1]$ such that
\begin{align*}
|\psi(x)-\psi(\tilde{x})|
&\le|x-\tilde{x}|\lf|\nabla\psi(x+\tz(\tilde{x}-x))\r|
\ls|x-\tilde{x}|(1+|x+\tz(\tilde{x}-x)|)^{-n-\epsilon}\\
&\ls|x-\tilde{x}|[(1+|x|)^{-n-\epsilon}+(1+|\tilde{x}|)^{-n-\epsilon}].
\end{align*}
Thus, up to a positive harmless constant multiple, $\psi\in \mathcal{C}_{(1,\,\epsilon)}(\rn)$
for any $\epsilon\in(0,\,\fz)$.
Thus, for any $f\in C_c^\fz(\rn)$ and $x\in\rn$,
$g_{-\Delta}(f)(x)\ls\wz{g}_{(1,\,\epsilon)}(f)(x)$
and
\begin{align}\label{eq-1.15}
S_{-\Delta,\,\az}(f)(x)\ls \wz{G}_{(1,\,\epsilon),\,\az}(f)(x).
\end{align}

To estimate $g_{-\Delta}$, by \cite[Exercise 6.4]{w08}, we know that,
for any $f\in C_c^\fz(\rn)$ and $x\in\rn$,
$\wz{g}_{(1,\,\epsilon)}(f)(x)\sim \wz{G}_{(1,\,\epsilon)}(f)(x)$,
where the implicit positive equivalence constants depend only on $n$ and $\epsilon$.
Moreover, from \cite[Theorem 6.3]{w08}, it follows that, for any $\epsilon'\in(1,\,\fz)$,
$f\in C_c^\fz(\rn)$ and $x\in\rn$,
$\wz{G}_{(1,\,\epsilon')}(f)(x)\ls G_1(f)(x)$. Thus, \eqref{eq-1.16} holds true.

To estimate $S_{\Delta,\,\az}$, by \cite[Lemma 6.2]{w08} and an argument similar to
that used in the proof of \cite[Exercise 6.1]{w08}, we know that,
for any $s\in(0,\,\fz)$, $f\in C_c^\fz(\rn)$, $x\in\rn$ and $y\in B(x,\,s)$,
\begin{align*}
\wz{A}_{(1,\,\epsilon)}(f)\lf(y,\,\frac{s}{\az}\r)
\le \az^{n+1}\wz{A}_{(1,\,\epsilon')}(f)(y,\,s),
\end{align*}
where $\epsilon':=\epsilon-1$.
By this and \cite[Theorem 6.3]{w08}, we find that, for any $f\in C_c^\fz(\rn)$ and $x\in\rn$,
\begin{align*}
\wz{G}_{(1,\,\epsilon),\,\az}(f)(x)
&=\lf\{\int_0^\fz\int_{|x-y|<\az t}\lf[\wz{A}_{(1,\,\epsilon)}(f)(y,\,t)\r]^2\,\dydt\r\}^{1/2}\\
&\sim\az^{n/2}\lf\{\int_0^\fz\int_{|x-y|<s}
\lf[\wz{A}_{(1,\,\epsilon)}(f)\lf(y,\,\frac{s}{\az}\r)\r]^2\,\frac{dy\,ds}{s^{n+1}}\r\}^{1/2}\\
&\ls\az^{\frac{3n}{2}+1}\lf\{\int_0^\fz\int_{|x-y|<s}
\lf[\wz{A}_{(1,\,\epsilon')}(f)(y,\,s)\r]^2\,\frac{dy\,ds}{s^{n+1}}\r\}^{1/2}\\
&\ls\az^{\frac{3n}{2}+1}\wz{G}_{(1,\,\epsilon')}(f)(x)
\ls\az^{\frac{3n}{2}+1}G_1(f)(x),
\end{align*}
where $\epsilon'\in(1,\,\fz)$.
This, combined with \eqref{eq-1.15}, implies that \eqref{eq-1.14} holds true,
which completes the proof of Lemma \ref{lem-9}.
\end{proof}

Let $\rho$ be as in \eqref{eq aux} with $V\in RH_q(\rn)$ and $q\in(n/2,\,\fz)$, and
\begin{align}\label{eq-ball}
\mathcal{B}_{\rho}:=\{B(x,\,r):\ x\in\rn,\ r\le\rho(x)\}.
\end{align}
The following $\rho$-localized weights were introduced in \cite{bhs11}.
\begin{definition}[\cite{bhs11}]
Let $n\geq 3$, $\rho$ be as in \eqref{eq aux} with $V\in RH_q(\rn)$ and $q\in(n/2,\,\fz)$,
and $p\in(1,\,\fz)$.
The \emph{local weight class $A_p^{\rho,\,\loc}(\rn)$} is defined to be the set
of all non-negative locally integrable functions $w$ on $\rn$ such that
\begin{align}\label{eq-loc1}
[w]_{A_p^{\rho,\,\loc}(\rn)}:=\sup_{B\in{\cb_\rho}}\frac{1}{|B|}\int_B w(x)\,dx
\lf\{\frac{1}{|B|}\int_B [w(x)]^{-\frac{1}{p-1}}\,dx\r\}^{p-1}<\fz.
\end{align}
\end{definition}

\begin{remark}\label{rem-2}
\begin{enumerate}
\item[(i)] Let $n\geq 3$, $\rho$ be as in \eqref{eq aux} with $V\in RH_q(\rn)$ and $q\in(n/2,\,\fz)$,
$p\in(1,\,\fz)$, $\tz\in[0,\,\fz)$ and $w\in A_p^{\rho,\,\tz}(\rn)$.
Then it is easy to see that $A_p^{\rho,\,\tz}(\rn)\st A_p^{\rho,\,{\loc}}(\rn)$
and there exists a positive constant $C$ such that, for any $w\in A_p^{\rho,\,\tz}(\rn)$,
$$[w]_{A_p^{\rho,\,{\loc}}(\rn)}\le C[w]_{A_p^{\rho,\,\tz}(\rn)}.$$

\item[(ii)]
Let $n\geq 3$, $\rho$ be as in \eqref{eq aux} with $V\in RH_q(\rn)$ and $q\in(n/2,\,\fz)$,
$p\in(1,\,\fz)$ and $\beta\in(1,\,\fz)$.
Define $A_p^{\beta\rho,\,\loc}(\rn)$
to be the set of all non-negative locally integrable functions $w$ on $\rn$
satisfying \eqref{eq-loc1} with $B\in\cb_\rho$ therein replaced by
$B\in\{B(x,\,r):\ x\in\rn,\,r\le\beta\rho(x)\}$.
Then, from \cite[Corllary 1]{bhs11} and its proof, we deduce that, for any $\beta\in(1,\,\fz)$,
$A_p^{\rho,\,\loc}(\rn)=A_p^{\beta\rho,\,\loc}(\rn)$.
Moreover, for any $w\in A_p^{\rho,\,\tz}(\rn)$,
$$[w]_{A_p^{\rho,\,\loc}(\rn)}\sim [w]_{A_p^{\beta\rho,\,\loc}(\rn)},$$
where the implicit positive equivalence constants depend only on $\beta,\,n$ and $\rho$.

\item[(iii)]
Let $B_0$ be a ball in $\rn$.
A weight $w$, defined on $B_0$, is said to belong to $A_p(B_0)$ if
the inequality \eqref{eq-loc1} holds true for every ball $B\st B_0$.
\end{enumerate}
\end{remark}

The following lemma concerns the extension of weights, which is just \cite[Lemma 1]{bhs11}.
\begin{lemma}[\cite{bhs11}]\label{lem-7}
Let $B_0$ be a ball in $\rn$.
Assume that $p\in(1,\,\fz)$ and $w_0\in A_p(B_0)$. Then $w_0$ has an extension $w\in A_p(\rn)$ on $\rn$
such that, for any $x\in B_0$, $w_0(x)=w(x)$ and
$[w_0]_{A_{p}(B_0)}\sim[w]_{A_p(\rn)}$, where the implicit positive equivalence
constants are independent of $w_0$ and $p$.
\end{lemma}

The following lemma is just \cite[Proposition 2]{bhs11} (see also \cite[Lemma 2.3]{dz99}).
\begin{lemma}[\cite{bhs11}]\label{lem-1}
There exists a sequence $\{x_j\}_{j\in\nn}$ of points in $\rn$ such that
the family $\{B_j:=B(x_j,\,\rho(x_j))\}_{j\in\nn}$
of balls satisfies that
\begin{enumerate}
\item[(i)] $\cup_{j\in\nn}B_j=\rn$;

\item[(ii)] For any $\sigma\in[1,\,\fz)$, there exist positive
constants $C$ and $N$ such that, for any $x\in\rn$,
$\sum_{j\in\nn}{\mathbf 1}_{\sigma B_j}(x)\le C\sigma^{N}$.
\end{enumerate}
\end{lemma}

In what follows, for any $x\in\rn$, let
\begin{align}\label{eq-rb}
B_x:=B(x,\,\rho(x)).
\end{align}
For any $\az\in(0,\,\fz)$, we define the following $\rho$-localized
$R_{-\Delta}$, $M^\ast_{-\Delta,\,\az}$, $g_{-\Delta}$ and $S_{-\Delta,\,\az}$, respectively, by
setting, for any $f\in L^1_{\loc}(\rn)$ and $x\in\rn$,
\begin{align}\label{eq-loc2}
R^{\loc}_{-\Delta}(f)(x):=R_{-\Delta}\lf(f{\mathbf 1}_{B_x}\r)(x),
\end{align}
\begin{align}\label{eq-loc4}
M^{\ast,\,\loc}_{-\Delta,\,\az}(f)(x):=M^\ast_{-\Delta,\,\az}\lf(f{\mathbf 1}_{2B_x}\r)(x),
\end{align}
\begin{align}\label{eq-loc3}
g_{-\Delta}^{\loc}(f)(x):=g_{-\Delta}\lf(f{\mathbf 1}_{B_x}\r)(x)
\end{align}
and
\begin{align}\label{eq-sl}
S_{-\Delta,\,\az}^{\loc}(f)(x):=S_{-\Delta,\,\az}\lf(f{\mathbf 1}_{2B_x}\r)(x).
\end{align}

Applying the method used to prove the boundedness of $\rho$-localized operators
(see \cite[Proposition 4]{bhs11}),
we obtain the following quantitative weighted estimates of $R_{-\Delta}^{\loc}$,
$M^{\ast,\,\loc}_{-\Delta,\,\az}$, $g_{-\Delta}^{\loc}$ and $S_{-\Delta,\,\az}^{\loc}$,
respectively.
\begin{lemma}\label{lem-2}
Let $n\geq 3$, $\rho$ be as in \eqref{eq aux} with $V\in RH_q(\rn)$ and $q\in(n/2,\,\fz)$,
$\az\in[1,\,\fz)$ and $p\in(1,\,\fz)$.
Then there exists a positive constant $C$ such that,
for any $w\in A_p^{\rho,\,{\loc}}(\rn)$ and $f\in L^p(w)$,
\begin{align}\label{eq-20}
\lf\|R^{\loc}_{-\Delta}(f)\r\|_{L^p(w)}
\le C[w]^{\frac{1}{p-1}}_{A_p^{\rho,\,{\loc}}(\rn)}\|f\|_{L^p(w)},
\end{align}
\begin{align}\label{eq-21}
\lf\|M^{\ast,\,\loc}_{-\Delta,\,\az}(f)\r\|_{L^p(w)}
\le C\az^n[w]^{\frac{1}{p-1}}_{A_p^{\rho,\,{\loc}}(\rn)}\|f\|_{L^p(w)},
\end{align}
\begin{align}\label{eq-23}
\lf\|g_{-\Delta}^{\loc}(f)\r\|_{L^p(w)}
\le C[w]^{\max\{\frac12,\,\frac{1}{p-1}\}}_{A_p^{\rho,\,{\loc}}(\rn)}
\|f\|_{L^p(w)}
\end{align}
and
\begin{align}\label{eq-22}
\lf\|S_{-\Delta,\,\az}^{\loc}(f)\r\|_{L^p(w)}
\le C\az^{\frac{3n}{2}+1}[w]^{\max\{\frac12,\,\frac{1}{p-1}\}}_{A_p^{\rho,\,{\loc}}(\rn)}
\|f\|_{L^p(w)}.
\end{align}
\end{lemma}

\begin{proof}
We first prove \eqref{eq-21}.
By Lemma \ref{lem-1}, we know that there exists
a sequence $\{x_j\}_{j\in\nn}$ of points in $\rn$ such that
the family $\{B_j:= B(x_j,\,\rho(x_j))\}_{j\in\nn}$ of balls satisfies (i) and (ii) of Lemma \ref{lem-1}.
Let $\sigma:=1+C 2^{1+\frac{N_0}{N_0+1}}$ with positive constants $C$ and $N_0$ same as in Lemma \ref{lem aux}.
For any $j\in\nn$, let $\wz{B}_j:=\sigma B_j$.
By Lemma \ref{lem aux}, we know that,
for any $j\in\nn$ and $x\in B_j$, $2B_x=B(x,\,2\rho(x))\st\wz{B}_j$.
By this, it is easy to see that, for any $f\in L^p(w)$ and $x\in\rn$,
\begin{align}\label{eq-1.9}
M_{-\Delta,\,\az}^{\ast,\,\loc}(f)(x)\le \sum_{j=1}^\fz{\mathbf 1}_{B_j}(x)
M_{-\Delta,\,\az}^\ast(|f|{\mathbf 1}_{\wz{B}_j})(x)=:\wz{M}_\az^\ast(|f|)(x).
\end{align}

Let $w\in  A_p^{\rho,\,{\loc}}(\rn)$.
Then we know that, for any $j\in\nn$,
\begin{align}\label{eq-1.11}
w|_{\wz{B}_j}\in A_p(\wz{B}_j)\quad
\text{and}\quad [w|_{\wz{B}_j}]_{A_p(\wz{B}_j)}\ls[w]_{A_p^{\rho,\,{\loc}}(\rn)},
\end{align}
where the implicit positive constant is independent of $B_j$.
Indeed, for any ball $B:=B(x_B,\,r_B)\st \wz{B}_j$, if $r_B\le\sigma\rho(x_B)$, then, by the fact that $w\in  A_p^{\rho,\,{\loc}}(\rn)$ and Remark \ref{rem-2}(ii),
we find that $w\in A_p^{\sigma\rho,\,{\loc}}(\rn)$ with
$[w]_{A_p^{\rho,\,{\loc}}(\rn)}\sim[w]_{A_p^{\sigma\rho,\,{\loc}}(\rn)}$.
From this, we further deduce that
\begin{align}\label{eq-1.8}
\frac{1}{|B|}\int_B w(y)\,dy\lf\{\frac{1}{|B|}\int_B [w(y)]^{-\frac{1}{p-1}}\,dy\r\}^{p-1}\le [w]_{A_p^{\sigma\rho,\,{\loc}}(\rn)}\sim[w]_{A_p^{\rho,\,{\loc}}(\rn)}.
\end{align}
If $r_B> \sigma\rho(x_B)$,
then, we have $B(x_B,\,\sigma\rho(x_B))\st B(x_B,\,r_B)\st \wz{B}_j$.
By the fact that $|x_j-x_B|<\rho(x_j)$ and
Lemma \ref{lem aux}, we obtain
$\rho(x_B)\sim \rho(x_j)$. Thus, $|B(x_B,\,r_B)|\sim|\wz{B}_j|$.
From this, it follows that
\begin{align*}
\frac{1}{|B|}\int_B w(y)\,dy\lf\{\frac{1}{|B|}\int_B [w(y)]^{-\frac{1}{p-1}}\,dy\r\}^{p-1}
&\ls\frac{1}{|\wz{B}_j|}\int_{\wz{B}_j}w(y)\,dy\lf\{\frac{1}{|\wz{B}_j|}
\int_{\wz{B}_j}[w(y)]^{-\frac{1}{p-1}}\,dy\r\}^{p-1}\\
&\ls[w]_{A_p^{\sigma\rho,\,{\loc}}(\rn)}\sim[w]_{A_p^{\rho,\,{\loc}}(\rn)}.
\end{align*}
This, combined with \eqref{eq-1.8}, implies \eqref{eq-1.11}.
By Lemma \ref{lem-7}, we further know that, for any $j\in\nn$, $w|_{\wz{B}_j}$
admits an extension $w_j$ on $\rn$ such that $w_j\in A_p(\rn)$ and
$$[w_j]_{A_p(\rn)}\sim[w|_{\wz{B}_j}]_{A_p(\wz{B}_j)}\ls[w]_{A_p^{\rho,\,{\loc}}(\rn)},$$
where the implicit positive constants are independent of $j$.
From this, Lemma \ref{lem-1}(ii), \eqref{eq-0.2} and \eqref{eq-0.3}, it follows that,
for any $w\in A_p^{\rho,\,\loc}(\rn)$ and $f\in L^p(w)$,
\begin{align*}
\lf\|\wz{M}_\az^\ast(|f|)\r\|_{L^p(w)}^p
&=\int_\rn \sum_{j\in\nn}{\mathbf 1}_{B_j}(x)\lf|M_{-\Delta,\,\az}^\ast(|f|{\mathbf 1}_{\wz{B}_j})(x)\r|^pw(x)\,dx\\
&\ls\sum_{j\in\rn}\int_{B_j}
\lf|M_{-\Delta,\,\az}^\ast(|f|{\mathbf 1}_{\wz{B}_j})(x)\r|^pw(x)\,dx
\ls\sum_{j\in\nn}\int_{\rn}\lf|M_{-\Delta,\,\az}^\ast(|f|{\mathbf 1}_{\wz{B}_j})(x)\r|^p w_j(x)\,dx\noz\\
&\ls\az^{np}\sum_{j\in\nn}\int_{\rn}\lf|M(|f|{\mathbf 1}_{\wz{B}_j})(x)\r|^p w_j(x)\,dx
\ls\az^{np}\sum_{j\in\nn}[w_j]_{A_p(\rn)}^{\frac{p}{p-1}}\int_{\wz{B}_j}|f(x)|^pw_j(x)\,dx\noz\\
&\ls\az^{np}[w]_{A^{\rho,\,{\loc}}_p(\rn)}^{\frac{p}{p-1}}\int_\rn|f(x)|^p w(x)\,dx.\noz
\end{align*}
This, together with \eqref{eq-1.9}, implies that
\eqref{eq-21} holds true.
By an argument similar to that used in the proof of \eqref{eq-21},
we also know that \eqref{eq-20} holds true.

To prove \eqref{eq-23} and \eqref{eq-22},
we only need to repeat the proof of \eqref{eq-21} via replacing
\eqref{eq-0.2} and \eqref{eq-0.3} therein by Lemma \ref{lem-9} and \eqref{eq-0.0},
respectively, and we omit the details.
This finishes the proof of Lemma \ref{lem-2}.
\end{proof}

\section{Proofs of Theorems \ref{thm-1} and \ref{thm-2}}\label{s-2}

In this section, we show Theorems \ref{thm-1} and \ref{thm-2}.

\begin{proof}[Proof of Theorem \ref{thm-1}]
We show this theorem by three steps.

\emph{Step 1.} In this step, we show that, for any given $p\in(1,\,\fz)$ and $\tz\in[0,\,\fz)$,
and for any $w\in A_p^{\rho,\,\tz}(\rn)$ and $f\in C_c^\fz(\rn)$,
\begin{align}\label{eq-3.5}
\|g_L(f)\|_{L^p(w)}\ls [w]_{A_p^{\rho,\,\tz}(\rn)}^{\max\{\frac12,\,\frac{1}{p-1}\}}\|f\|_{L^p(w)}.
\end{align}
For any $f\in C_c^\fz(\rn)$ and $x\in\rn$, we write
\begin{align*}
g_L(f)(x):=g_L(f{\mathbf 1}_{B_x})(x)+g_L\lf(f{\mathbf 1}_{B_x^\com}\r)(x)=:g_{L}^{\loc}(f)(x)+g_{L}^{\rm glob}(f)(x),
\end{align*}
where $B_x$ is as in \eqref{eq-rb}.

We first estimate $g_{L}^{\rm glob}$. By Lemma \ref{lem-8}, we find that,
for any $f\in C_c^\fz(\rn)$ and $x\in\rn$,
\begin{align}\label{eq-3.3}
g_{L}^{\rm glob}(f)(x)
&=\lf[\int_0^\fz \lf|tLe^{-tL}\lf(f{\mathbf 1}_{B_x}^\com\r)(x)\r|^2\,\frac{dt}{t}\r]^{1/2}\\
&=\lf[\int_0^\fz \lf|\int_{|x-y|>\rho(x)}t\frac{\partial}{\partial t}K_t(x,\,y)f(y)\,dy\r|^2
\,\frac{dt}{t}\r]^{1/2}\noz\\
&\ls\lf[\int_0^\fz \lf\{\frac{1}{t^{n/2}}\lf[1+\frac{\sqrt{t}}{\rho(x)}\r]^{-N}\int_{|x-y|>\rho(x)}
e^{-c\frac{|x-y|^2}{t}}|f(y)|\,dy\r\}^2\,\frac{dt}{t}\r]^{1/2}\noz\\
&\ls\lf\{\int_0^\fz \frac{1}{t^n}\lf[1+\frac{\sqrt{t}}{\rho(x)}\r]^{-2N}\r.\noz\\
&\hs\times\lf.\lf[\sum_{j=0}^\fz e^{-c\frac{[2^j\rho(x)]^2}{t}}
\int_{B(x,\,2^{j+1}\rho(x))\setminus B(x,\,2^j\rho(x))}|f(y)|\,dy\r]^2\,\frac{dt}{t}\r\}^{1/2}\noz\\
&\ls\lf[\int_0^\fz\frac{1}{t^n}\lf\{1+\frac{\sqrt{t}}{\rho(x)}\r\}^{-2N}
\lf\{\sum_{j=0}^\fz \lf[\frac{\sqrt{t}}{2^j\rho(x)}\r]^{A}2^{j\gamma}[2^{j}\rho(x)]^n\r.\r.\noz\\
&\hs\lf.\lf.\times\lf[1+\frac{2^{j+1}\rho(x)}{\rho(x)}\r]^{-\gamma}\frac{1}{|B(x,\,2^{j+1}\rho(x))|}
\int_{B(x,\,2^{j+1}\rho(x))}|f(y)|\,dy\r\}^2\,\frac{dt}{t}\r]^{1/2}\noz\\
&\ls\lf\{\int_0^\fz\lf[\frac{\rho(x)}{\sqrt{t}}\r]^{2n}\lf[1+\frac{\sqrt{t}}{\rho(x)}\r]^{-2N}
\lf[\frac{\sqrt{t}}{\rho(x)}\r]^{2A}
\lf[\sum_{j=0}^\fz 2^{-j(M-n-\gamma)}M_\gamma(f)(x)\r]^2\,\frac{dt}{t}\r\}^{1/2}\noz\\
&\ls\lf\{\int_0^\fz\lf[1+\frac{\sqrt{t}}{\rho(x)}\r]^{-2N}
\lf[\frac{\sqrt{t}}{\rho(x)}\r]^{2(A-n)}\,\frac{dt}{t}\r\}^{1/2}M_\gamma(f)(x)\noz\\
&\ls\lf[\int_0^\fz\frac{s^{2(A-n)-1}}{(1+s)^{2N}}\,ds\r]^{1/2}M_\gamma(f)(x)
\ls M_\gamma(f)(x),\noz
\end{align}
where $A,\,N\in(0,\,\fz)$ such that $A>n+\gamma$ and $N>A-n$,
$M_\gamma$ is as in \eqref{eq-fm} and
$\gamma\in(0,\,\fz)$ is fixed later.

Next, we estimate $g_{L}^{\loc}$. Indeed, we have,
for any $f\in C_c^\fz(\rn)$ and $x\in\rn$,
\begin{align}\label{eq-3.2}
g_L^{\loc}(f)(x)
&\le\lf\{\int_0^{[\rho(x)]^2}\lf|tLe^{-tL}\lf(f{\mathbf 1}_{B_x}\r)(x)\r|^2\,\frac{dt}{t}\r\}^{1/2}
+\lf\{\int_{[\rho(x)]^2}^\fz\cdots\,\frac{dt}{t}\r\}^{1/2}\\
&\le\lf\{\int_0^{[\rho(x)]^2}\lf|\lf(tLe^{-tL}-t\Delta e^{-t\Delta}\r)
\lf(f{\mathbf 1}_{B_x}\r)(x)\r|^2\,\frac{dt}{t}\r\}^{1/2}\noz\\
&\hs+\lf\{\int_0^{[\rho(x)]^2}\lf|t\Delta e^{-t\Delta}
\lf(f{\mathbf 1}_{B_x}\r)(x)\r|^2\,\frac{dt}{t}\r\}^{1/2}
+\lf\{\int_0^{[\rho(x)]^2}\lf|tLe^{-tL}\lf(f{\mathbf 1}_{B_x}\r)(x)\r|^2\,\frac{dt}{t}\r\}^{1/2}\noz\\
&=:{\rm K}_1(x)+{\rm K}_2(x)+{\rm K}_3(x).\noz
\end{align}

For ${\rm K}_2$, it is easy to see that, for any $f\in C_c^\fz(\rn)$ and $x\in\rn$,
\begin{align}\label{eq-3.1}
K_2(x)\le g_{-\Delta}^{\loc}(f)(x),
\end{align}
where $g_{-\Delta}^{\loc}$ is as in \eqref{eq-loc3}.

For ${\rm K}_1$, by the perturbation formula as in \cite[p.\,578]{bhs11}
(see also \cite[Proposition 5.1]{dzi05}),
we find that there exists some $\delta\in(0,\,\fz)$
such that, for any $f\in C_c^\fz(\rn)$ and $x\in\rn$,
\begin{align}\label{eq-3.4}
\lf|\lf(tL e^{-tL}-t\Delta e^{-t\Delta}\r)(f{\mathbf 1}_{B_x})(x)\r|
\ls\int_{B_x}\frac{1}{t^{n/2}}\lf\{\lf[\frac{\sqrt{t}}{\rho(x)}\r]^{\delta}
+\lf[\frac{\sqrt{t}}{\rho(y)}\r]^{\delta}\r\}e^{-c\frac{|x-y|^2}{t}}|f(y)|\,dy.
\end{align}
Moreover, by the fact that $|x-y|<\rho(x)$ and Lemma \ref{lem aux}, we know that $\rho(x)\sim\rho(y)$.
From this and \eqref{eq-3.4}, it follows that, for any $f\in C_c^\fz(\rn)$ and $x\in\rn$,
\begin{align}\label{eq-3.0}
{\rm K}_1(x)
&\ls\lf[\int_0^{[\rho(x)]^2}\lf\{\int_{B_x}\frac{1}{t^{n/2}}
\lf[\frac{\sqrt{t}}{\rho(x)}\r]^{\delta}e^{-c\frac{|x-y|^2}{t}}|f(y)|\,dy\r\}^2\,\frac{dt}{t}\r]^{1/2}\\
&\ls\lf\{\int_0^{[\rho(x)]^2}\lf[\frac{\sqrt{t}}{\rho(x)}\r]^{2\delta}\,\frac{dt}{t}\r\}^{1/2}
R_{-\Delta}^{\loc}(|f|)(x)
\ls R_{-\Delta}^{\loc}(|f|)(x),\noz
\end{align}
where $R_{-\Delta}^{\loc}$ is as in \eqref{eq-loc2}.

For ${\rm K}_3$, by Lemma \ref{lem-8}, we obtain,
for any $f\in C_c^\fz(\rn)$ and $x\in\rn$,
\begin{align*}
{\rm K}_3(x)
&\ls\lf[\int_{[\rho(x)]^2}^\fz\lf\{\int_{B_x}\lf[1+\frac{\sqrt{t}}{\rho(x)}\r]^{-N}
\frac{1}{t^{n/2}}e^{-c\frac{|x-y|^2}{t}}|f(y)|\,dy\r\}^2\,\frac{dt}{t}\r]^{1/2}\\
&\ls\lf\{\int_{[\rho(x)]^2}^\fz\lf[1+\frac{\sqrt{t}}{\rho(x)}\r]^{-2N}\,\frac{dt}{t}\r\}^{1/2}
R_{-\Delta}^{\loc}(|f|)(x)
\ls R_{-\Delta}^{\loc}(|f|)(x).
\end{align*}
From this, \eqref{eq-3.2}, \eqref{eq-3.1} and \eqref{eq-3.0},
we deduce that, for any $f\in C_c^\fz(\rn)$ and $x\in\rn$,
$$g_{L}^{\loc}(f)(x)\ls g_{-\Delta}^{\loc}(f)(x)+R_{-\Delta}^{\loc}(|f|)(x).$$
This, combined with \eqref{eq-3.3}, implies that,
for any $f\in C_c^\fz(\rn)$ and $x\in\rn$,
$$g_{L}(f)(x)\ls M_\gamma(f)(x)+g_{-\Delta}^{\loc}(f)(x)+R_{-\Delta}^{\loc}(|f|)(x).$$

Let $\tz\in[0,\,\fz)$.
From Remark \ref{rem-2}(i), Lemmas \ref{lem-10} and \ref{lem-2}, it follows that,
for any $w\in A_3^{\rho,\,\tz}(\rn)$ and $f\in L^3(w)$,
\begin{align*}
\|g_L(f)\|_{L^3(w)}
&\ls \|M_\gamma(f)\|_{L^3(w)}+\lf\|g_{-\Delta}^{\loc}(f)\r\|_{L^3(w)}
+\lf\|R_{-\Delta}^{\loc}(|f|)\r\|_{L^3(w)}\\
&\ls[w]_{A_3^{\rho,\,\tz}(\rn)}^{1/2}\|f\|_{L^3(w)}+[w]_{A_3^{\rho,\,\loc}(\rn)}^{1/2}\|f\|_{L^3(w)}
\ls[w]_{A_3^{\rho,\,\tz}(\rn)}^{1/2}\|f\|_{L^3(w)},
\end{align*}
where we fix $\gamma:=\frac{3\tz}{2}$. This, together with Lemma \ref{lem-4}
and the fact that $C_c^\fz(\rn)$ is dense in $L^p(w)$,
implies that \eqref{eq-3.5} holds true.

\emph{Step 2}. In this step, we show that, for any given $p\in(1,\,\fz)$ and $\tz\in[0,\,\fz)$,
and for any $w\in A_p^{\rho,\,\tz}(\rn)$ and $f\in C_c^\fz(\rn)$,
\begin{align}\label{eq-3.6}
\lf\|g_{L,\,\lz}^{\ast}(f)\r\|_{L^p(w)}
\ls[w]_{A_p^{\rho,\,\tz}(\rn)}^{\max\{\frac12,\,\frac{1}{p-1}\}}\|f\|_{L^p(w)}.
\end{align}
Indeed, for any $f\in C_c^\fz(\rn)$ and $x\in\rn$, we have
\begin{align*}
\lf[g_{L,\,\lambda}^\ast(f)(x)\r]^2
&=\int_0^\fz\int_\rn\lf(\frac{t}{t+|x-y|}\r)^{\lambda n}
\lf|t^2L e^{-t^2L}(f)(y)\r|^2\,\dydt\\
&=\int_0^\fz\int_{B(x,\,t)}\lf(\frac{t}{t+|x-y|}\r)^{\lambda n}\lf|t^2L e^{-t^2L}(f)(y)\r|^2\,\dydt\\
&\hs+\sum_{j=1}^\fz\int_0^\fz\int_{B(x,\,2^jt)\setminus B(x,\,2^{j-1}t)}\ldots\,\dydt\\
&\ls\lf[S_L(f)(x)\r]^2+\sum_{j=1}^\fz2^{-j\lambda n}\lf[S_{L,\,2^j}(f)(x)\r]^2,
\end{align*}
namely,
\begin{align}\label{eq-2.1}
g_{L,\,\lambda}^\ast(f)(x)
\ls \sum_{j=0}^\fz 2^{-j\lambda n/2}S_{L,\,2^j}(f)(x),
\end{align}
where $S_{L,\,2^j}$ is as in \eqref{eq-0.5} with $\az:=2^j$ therein.

Next, we estimate $S_{L,\,2^j}$. To this end, let $\az\in[1,\,\fz)$.
For any $f\in C_c^\fz(\rn)$ and $x\in\rn$, we write
\begin{align}\label{eq-2.8}
S_{L,\,\az}(f)(x)
=S_{L,\,\az}(f{\mathbf 1}_{2B_x})(x)+S_{L,\,\az}\lf(f{\mathbf 1}_{(2B_x)^\com}\r)(x)
=:S_{L,\,\az}^{\loc}(f)(x)+S_{L,\,\az}^{\rm glob}(f)(x),
\end{align}
where $B_x$ is as in \eqref{eq-rb}.

To deal with $S_{L,\,\az}^{\rm glob}(f)$, we have,
for any $f\in C_c^\fz(\rn)$ and $x\in\rn$,
\begin{align*}
S_{L,\,\az}^{\rm glob}(f)(x)
&\le\lf\{\int_0^{\rho(x)/\az}\int_{|x-y|<\az t}\lf|t^2L e^{t^2L}(f{\mathbf 1}_{(2B_x)^\com})(y)
\r|^2\,\dydt\r\}^{\frac12}
+\lf\{\int_{\rho(x)/\az}^\fz\int_{|x-y|<\az t}\cdots\,\dydt\r\}^{\frac12}\\
&={\rm I}(x)+{\rm II}(x).
\end{align*}
First, we estimate ${\rm I}$. Indeed, for any $f\in C_c^\fz(\rn)$ and $x\in\rn$,
it holds true that
\begin{align}\label{eq-1.1}
{\rm I}(x)
&=\lf\{\int_0^{\rho(x)/\az}\int_{|x-y|<\az t}\lf|
\int_\rn t^2\frac{\partial}{\partial s}K_s(y,\,z)\Big|_{s=t^2}f(z){\mathbf 1}_{(2B_x)^\com}(z)\,dz\r|^2\,\dydt\r\}^{1/2}\\
&\le\lf\{\int_0^{\rho(x)/\az}\int_{|x-y|<\az t}\lf[\sum_{k=2}^\fz
\int_{2^k B_x\setminus 2^{k-1}B_x}
\lf|t^2\frac{\partial}{\partial s}K_s(y,\,z)\Big|_{s=t^2}\r||f(z)|\,dz\r]^2\,\dydt\r\}^{1/2},\noz
\end{align}
where $K_s(\cdot,\,\cdot)$ denotes the integral kernel
of $e^{-sL}$ as in Lemma \ref{lem-8}.
For any $z\in(2^kB_x)\setminus (2^{k-1}B_x)$ and $y\in B(x,\,\az t)$, it is easy to see that
\begin{align}\label{eq-1.1x}
|y-z|\geq |x-z|-|y-x|\geq 2^{k-1}\rho(x)-\az t
\geq 2^{k-1}\rho(x)-\rho(x)\geq 2^{k-2}\rho(x).
\end{align}
By the fact that $\az\in[1,\,\fz)$ and Lemma \ref{lem aux}, we know that, for any $y\in B(x,\,\az t)$,
\begin{align}\label{eq-1.1y}
\lf[1+\frac{t}{\rho(y)}\r]^{-1}
&\le\lf[\frac1\az+\frac{t}{\rho(y)}\r]^{-1}
=\az\lf[1+\frac{\az t}{\rho(y)}\r]^{-1}\\
&\ls\az\lf\{1+\frac{\az t}{\rho(x)}\lf[1+\frac{|x-y|}{\rho(x)}\r]^{-\frac{N_0}{N_0+1}}\r\}^{-1}
\ls\az\lf[1+\frac{\az t}{\rho(x)}\r]^{-\frac{1}{N_0+1}},\noz
\end{align}
where $N_0\in(0,\,\fz)$ is as in Lemma \ref{lem aux}.
From \eqref{eq-1.1y}, \eqref{eq-1.1x} and Lemma \ref{lem-8}, we deduce that,
for any $f\in C_c^\fz(\rn)$ and $x\in\rn$,
\begin{align*}
&\sum_{k=2}^\fz
\int_{2^k B_x\setminus 2^{k-1}B_x}
\lf|t^2\frac{\partial}{\partial s}K_s(y,\,z)\Big|_{s=t^2}\r||f(z)|\,dz\\
&\hs\ls\sum_{k=2}^\fz t^{-n}\lf[1+\frac{t}{\rho(y)}\r]^{-N}
\int_{2^k B_x\setminus 2^{k-1}B_x}e ^{-c\frac{[2^k\rho(x)]^2}{t^2}}|f(z)|\,dz\\
&\hs\ls\az^N\sum_{k=2}^\fz t^{-n}\lf[1+\frac{\az t}{\rho(x)}\r]^{-N/(N_0+1)}
\int_{2^k B_x\setminus 2^{k-1}B_x}\lf[\frac{t}{2^k\rho(x)}\r]^{A}|f(z)|\,dz\\
&\hs\ls\az^N\sum_{k=2}^\fz 2^{-k(A-n-\gamma)}\lf[\frac{t}{\rho(x)}\r]^{A-n}
\lf[1+\frac{\az t}{\rho(x)}\r]^{-N/(N_0+1)}\lf[1+\frac{2^k\rho(x)}{\rho(x)}\r]^{-\gamma}
\frac{1}{|2^kB_x|}\int_{2^kB_x}|f(z)|\,dz\\
&\hs\ls\az^N\lf[\frac{t}{\rho(x)}\r]^{A-n}
\lf[1+\frac{\az t}{\rho(x)}\r]^{-N/(N_0+1)}M_{\gamma}(f)(x),
\end{align*}
where $N,\,\gamma\in (0,\,\fz)$ are fixed later,
$M_\gamma$ is as in \eqref{eq-fm}
and $A$ a constant such that $A>n+\gamma$.
This, together with \eqref{eq-1.1}, implies that, for any $f\in C_c^\fz(\rn)$ and $x\in\rn$,
\begin{align}\label{eq-1.10}
{\rm I}(x)
&\ls\az^{N-A+n}\lf[\int_0^{\rho(x)/\az}\int_{|x-y|<\az t}\lf\{\lf[\frac{\az t}{\rho(x)}\r]^{A-n}
\lf[1+\frac{\az t}{\rho(x)}\r]^{-N/(N_0+1)}M_{\gamma}(f)(x)\r\}^2\,\dydt\r]^{1/2}\\
&\ls\az^{N-A+\frac{3n}{2}}M_{\gamma}(f)(x)\lf[\int_0^{1} \frac{s^{2(A-n)-1}}{(1+s)^{2N/(N_0+1)}}\,ds\r]^{1/2}
\ls\az^{N-A+\frac{3n}{2}}M_{\gamma}(f)(x).\noz
\end{align}

For ${\rm II}$, we have, for any $f\in C_c^\fz(\rn)$ and $x\in\rn$,
\begin{align}\label{eq-1.13}
{\rm II}(x)
&=\lf\{\int_{\rho(x)/\az}^\fz\int_{|x-y|<\az t}\lf|
\int_\rn t^2\frac{\partial}{\partial s}K_s(y,\,z)\Big|_{s=t^2}f(z){\mathbf 1}_{(2B_x)^\com}(z)\,dz\r|^2\,\dydt\r\}^{1/2}\\
&\le\lf\{\int_{\rho(x)/\az}^\fz\int_{|x-y|<\az t}\lf[\sum_{k=2}^\fz
\int_{2^k B_x\setminus 2^{k-1}B_x}
\lf|t^2\frac{\partial}{\partial s}K_s(y,\,z)\Big|_{s=t^2}\r||f(z)|\,dz\r]^2\,\dydt\r\}^{1/2}.\noz
\end{align}
Now we consider two cases.

Case 1) $|y-z|<\az t$. In this case, we then obtain
$2\az t>|y-z|+|y-x|\geq|x-z|\geq 2^{k-1}\rho(x)$.
From this, Lemma \ref{lem-8} and \eqref{eq-1.1y}, we deduce that,
for any $f\in C_c^\fz(\rn)$ and $x\in\rn$,
\begin{align}\label{eq-1.12}
&\sum_{k=2}^\fz
\int_{2^k B_x\setminus 2^{k-1}B_x}
\lf|t^2\frac{\partial}{\partial s}K_s(y,\,z)\Big|_{s=t^2}\r||f(z)|\,dz\\
&\hs\ls\sum_{k=2}^\fz t^{-n}\lf[1+\frac{t}{\rho(y)}\r]^{-N}
\int_{2^k B_x\setminus 2^{k-1}B_x}|f(z)|\,dz\noz\\
&\hs\ls\az^N\sum_{k=2}^\fz t^{-n}\lf[1+\frac{\az t}{\rho(x)}\r]^{-\frac{N}{N_0+1}}
\int_{2^k B_x\setminus 2^{k-1}B_x}|f(z)|\,dz\noz\\
&\hs\ls\az^N\sum_{k=2}^\fz \lf[\frac{\rho(x)}{t}\r]^{n-\tau}\lf[\frac{\rho(x)}{t}\r]^{\tau}
2^{-k[\frac{N}{N_0+1}-n-\gamma]}\lf[1+\frac{2^k\rho(x)}{\rho(x)}\r]^{-\gamma}
\frac{1}{|2^k B_x|}\int_{2^kB_x}|f(z)|\,dz\noz\\
&\hs\ls\az^N\sum_{k=2}^\fz \lf[\frac{\rho(x)}{t}\r]^{\tau}
2^{-k[\frac{N}{N_0+1}-\tau-\gamma]}M_\gamma(f)(x)
\ls\az^{N+\tau}\lf[\frac{\rho(x)}{\az t}\r]^\tau M_\gamma(f)(x),\noz
\end{align}
where $M_\gamma$ is as in \eqref{eq-1.10}, $N>(N_0+1)(\gamma+\tau)$
and $\tau\in(0,\,n)$ is fixed later.

Case 2) $|y-z|\geq \az t$. In this case, we find that
$$2|y-z|\geq|y-z|+\az t\geq |y-z|+|x-y|\geq|x-z|\geq 2^{k-1}\rho(x).$$
From this, \eqref{eq-1.1y} and Lemma \ref{lem-8}, it follows that,
for any $f\in C_c^\fz(\rn)$ and $x\in\rn$,
\begin{align*}
&\sum_{k=2}^\fz
\int_{2^k B_x\setminus 2^{k-1}B_x}
\lf|t^2\frac{\partial}{\partial s}K_s(y,\,z)\Big|_{s=t^2}\r||f(z)|\,dz\\
&\hs\ls\sum_{k=2}^\fz t^{-n}\lf[1+\frac{t}{\rho(y)}\r]^{-N}
\int_{2^kB_x\setminus 2^{k-1}B_x} e^{-c\frac{|y-z|^2}{t^2}}|f(z)|\,dz\\
&\hs\ls\az^N\sum_{k=2}^\fz t^{-n}\lf[1+\frac{\az t}{\rho(x)}\r]^{-\frac{N}{N_0+1}}
\lf[\frac{t}{2^k\rho(x)}\r]^{A}\int_{2^kB_x}|f(z)|\,dz\\
&\hs\ls\az^N\sum_{k=2}^\fz\lf[1+\frac{\az t}{\rho(x)}\r]^{-\frac{N}{N_0+1}}
2^{-k(A-\gamma-n)}\lf[\frac{t}{\rho(x)}\r]^{A-n}
\lf[1+\frac{2^k\rho(x)}{\rho(x)}\r]^{-\gamma}\frac{1}{|2^kB_x|}\int_{2^kB_x}|f(z)|\,dz\\
&\hs\ls\az^{N-A+n}\lf[1+\frac{\az t}{\rho(x)}\r]^{-\frac{N}{N_0+1}}
\lf[\frac{\az t}{\rho(x)}\r]^{A-n}M_\gamma(f)(x),
\end{align*}
where $A>n+\gamma$.
From this, \eqref{eq-1.12} and \eqref{eq-1.13}, we deduce that,
for any $f\in C_c^\fz(\rn)$ and $x\in\rn$,
\begin{align*}
{\rm II}(x)
&\ls\az^{N+\tau}\lf[\int_{\rho(x)/\az}^\fz\int_{|x-y|<\az t}\lf\{\lf[\frac{\rho(x)}{\az t}\r]^{2\tau}
+\lf[\frac{\az t}{\rho(x)}\r]^{2(A-n)}
\lf[1+\frac{\az t}{\rho(x)}\r]^{-\frac{2N}{N_0+1}}\r\}\,\dydt\r]^{1/2}M_\gamma(f)(x)\\
&\ls\az^{N+\tau+\frac{n}{2}}\lf\{\int_1^\fz\lf[s^{-2\tau}+s^{2(A-n)}(1+s)^{-\frac{2N}{N_0+1}}\r]
\,\frac{ds}{s}\r\}^{1/2}M_\gamma(f)(x)
\ls \az^{N+\tau+\frac{n}{2}}M_\gamma(f)(x),
\end{align*}
where we choose $A:=n+\gamma+\tau$ and $N>(N_0+1)(A-n)=(N_0+1)(\gamma+\tau)$.
This, combined with \eqref{eq-1.10}, implies that, for any $f\in C_c^\fz(\rn)$ and $x\in\rn$,
\begin{align}\label{eq-2.7}
S_{L,\,\az}^{\rm glob}(f)(x)\ls \az^{N+\frac{3n}{2}}M_\gamma(f)(x).
\end{align}

Next, we estimate $S_{L,\,\az}^{\loc}(f)$.
For any $f\in C_c^\fz(\rn)$ and $x\in\rn$, write
\begin{align}\label{eq-2.4}
S_{L,\,\az}^{\loc}(f)(x)
&\le\lf\{\int_0^{\rho(x)/\az}\int_{|x-y|<\az t}\lf|t^2L e^{-t^2L}(f{\mathbf 1}_{2B_x})(y)\r|^2\,\dydt\r\}^{1/2}\\
&\hs+\lf\{\int_{\rho(x)/\az}^\fz\int_{|x-y|<\az t}\cdots\,\dydt\r\}^{1/2}\noz\\
&\le\lf\{\int_0^{\rho(x)/\az}\int_{|x-y|<\az t}
\lf|\lf(t^2L e^{-t^2L}-t^2\Delta e^{-t^2\Delta}\r)(f{\mathbf 1}_{2B_x})(y)\r|^2\,\dydt\r\}^{1/2}\noz\\
&\hs+\lf\{\int_0^{\rho(x)/\az}\int_{|x-y|<\az t}
\lf|t^2\Delta e^{-t^2\Delta}(f{\mathbf 1}_{2B_x})(y)\r|^2\,\dydt\r\}^{1/2}\noz\\
&\hs+\lf\{\int_{\rho(x)/\az}^\fz
\int_{|x-y|<\az t}\lf|t^2L e^{-t^2L}(f{\mathbf 1}_{2B_x})(y)\r|^2\,\dydt\r\}^{1/2}\noz\\
&=:{\rm J}_1(x)+{\rm J}_2(x)+{\rm J}_3(x).\noz
\end{align}

For ${\rm J}_2$, it is easy to see that, for any $f\in C_c^\fz(\rn)$ and $x\in\rn$,
\begin{align}\label{eq-2.3}
{\rm J}_2(x)\le S_{-\Delta,\,\az}^{\loc}(f)(x),
\end{align}
where $S_{-\Delta,\,\az}^{\loc}$ is as in \eqref{eq-sl}.

For ${\rm J}_1$, by \cite[Proposition 5.1]{dzi05},
we find that there exists some $\delta\in(0,\,\fz)$
such that, for any $f\in C_c^\fz(\rn)$, $x\in\rn$ and $y\in B(x,\,\az t)$,
\begin{align*}
\lf|\lf(t^2L e^{-t^2L}-t^2\Delta e^{-t^2\Delta}\r)(f{\mathbf 1}_{2B_x})(y)\r|
\ls\int_{2B_x}\frac{1}{t^n}\lf\{\lf[\frac{t}{\rho(y)}\r]^{\delta}
+\lf[\frac{t}{\rho(z)}\r]^{\delta}\r\}e^{-c\frac{|y-z|^2}{t^2}}|f(z)|\,dz.
\end{align*}
Noticing that $0<\az t<\rho(x)$,
$|x-y|<\az t$ and $|x-z|<2\rho(x)$, by Lemma \ref{lem aux},
we know that $\rho(x)\sim\rho(y)\sim\rho(z)$.
Hence, for any $f\in C_c^\fz(\rn)$ and $x\in\rn$, it holds true that
\begin{align}\label{eq-2.5}
{\rm J}_1(x)
&\ls\lf\{\int_0^{\rho(x)/\az}\int_{|x-y|<\az t}\lf|\int_{2B_x}
\frac{1}{t^n}\lf[\frac{t}{\rho(x)}\r]^{\delta}e^{-c\frac{|y-z|^2}{t^2}}|f(z)|\,dz
\r|^2\,\dydt\r\}^{1/2}\\
&\ls\az^{n/2}\lf\{\int_0^{\rho(x)/\az}\lf[\frac{t}{\rho(x)}\r]^{2\delta}\,\frac{dt}{t}\r\}^{1/2}
M_{-\Delta,\,\az}^{\ast,\,\loc}(|f|)(x)
\ls \az^{n/2}M_{-\Delta,\,\az}^{\ast,\,\loc}(|f|)(x),\noz
\end{align}
where $M_{-\Delta,\,\az}^{\ast,\,\loc}$ is as in \eqref{eq-loc4}.

For ${\rm J}_3$, by \eqref{eq-1.1y}, we have,
for any $f\in C_c^\fz(\rn)$ and $x\in\rn$,
\begin{align}\label{eq-2.6}
{\rm J}_3(x)
&\ls\lf\{\int_{\rho(x)/\az}^\fz\int_{|x-y|<\az t}\lf|\int_{2B_x}
\lf[1+\frac{t}{\rho(y)}+\frac{t}{\rho(z)}\r]^{-N}t^{-n}e^{-c\frac{|y-z|^2}{t^2}}
|f(z)|\,dz\r|^2\,\dydt\r\}^{1/2}\\
&\ls\lf\{\int_{\rho(x)/\az}^\fz\int_{|x-y|<\az t}
\lf[1+\frac{t}{\rho(y)}\r]^{-2N}\lf[M_{-\Delta,\,\az}^{\ast,\,\loc}(|f|)(x)\r]^2\,\dydt\r\}^{1/2}\noz\\
&\ls\lf\{\int_{\rho(x)/\az}^\fz\int_{|x-y|<\az t}
\az^{2N}\lf[1+\frac{\az t}{\rho(x)}\r]^{-\frac{2N}{N_0+1}}\,\dydt\r\}^{1/2}
M_{-\Delta,\,\az}^{\ast,\,\loc}(|f|)(x)\noz\\
&\ls \az^{N+\frac{n}{2}}M_{-\Delta,\,\az}^{\ast,\,\loc}(|f|)(x),\noz
\end{align}
where $N\in(0,\,\fz)$. Combining \eqref{eq-2.6}, \eqref{eq-2.5}, \eqref{eq-2.3} and \eqref{eq-2.4}, we find that,
for any $f\in C_c^\fz(\rn)$ and $x\in\rn$,
$$S_{L,\,\az}^{\loc}(f)(x)\ls \az^{N+\frac{n}{2}}
M_{-\Delta,\,\az}^{\ast,\,\loc}(|f|)(x)+ S_{-\Delta,\,\az}^{\loc}(f)(x).$$
From this, \eqref{eq-2.7} and \eqref{eq-2.8}, we deduce that, for any $\az\in[1,\,\fz)$,
$f\in C_c^\fz(\rn)$ and $x\in\rn$,
\begin{align}\label{eq-2.9}
S_{L,\,\az}(f)(x)\ls \az^{N+\frac{3n}{2}}M_\gamma(f)(x)+\az^{N+\frac{n}{2}}
M_{-\Delta,\,\az}^{\ast,\,\loc}(|f|)(x)+ S_{-\Delta,\,\az}^{\loc}(f)(x),
\end{align}
where $N>(N_0+1)(\gamma+\tau)$, $N_0\in(0,\,\fz)$ is as in Lemma \ref{lem aux},
$\gamma\in(0,\,\fz)$ and $\tau\in(0,\,n)$.

Let $\tz\in[0,\,\fz)$ and $w\in A_3^{\rho,\,\tz}(\rn)$.
Fix $\gamma:=\frac{3\tz}{2}$ in \eqref{eq-2.9}.
By the fact that $\lambda\in(3+\frac{2}{n}\max\{\frac{3\tz}{2}k_0,\,1\},\,\fz)$,
$k_0:=\max\{\frac{\log_2{C_0}+2-n}{2-n/q},\,1\}$,
and Remark \ref{rem-7}, we could choose some $N_0\in(0,\,\fz)$ in Lemma \ref{lem aux} such that
$\lambda\in(3+\frac{2}{n}\max\{(N_0+1)\frac{3\tz}{2},\,1\},\,\fz).$
Then we could fix some $\tau\in(0,\,n)$ small enough and $N>(N_0+1)(\tau+\gamma)$ such that
$\lambda\in(3+\frac{2}{n}\max\lf\{N,\,1\r\},\,\fz)$.
From this, \eqref{eq-2.9}, \eqref{eq-2.1}, Lemma \ref{lem-10},
Remark \ref{rem-2}(i) and Lemma \ref{lem-2},
we deduce that, for any $f\in C_c^\fz(\rn)$,
\begin{align*}
\lf\|g_{L,\,\lambda}^\ast(f)\r\|_{L^3(w)}
&\ls\sum_{j=0}^\fz 2^{-j\lambda n/2}\lf\|S_{L,\,2^j}(f)\r\|_{L^3(w)}\\
&\ls\sum_{j=0}^\fz 2^{-j\lambda n/2}
\lf\{2^{j(N+\frac{3n}{2})}\lf\|M_\gamma(f)\r\|_{L^3(w)}
+2^{j(N+\frac{n}{2})}\lf\|M_{-\Delta,\,\az}^{\ast,\,\loc}(f)\r\|_{L^3(w)}
+\lf\|S_{-\Delta,\,\az}^{\loc}(f)\r\|_{L^3(w)}\r\}\\
&\ls\sum_{j=0}^\fz 2^{-j\lambda n/2}\lf\{
2^{j(N+\frac{3n}{2})}[w]_{A_3^{\rho,\,\tz}(\rn)}^{1/2}
+2^{j(N+\frac{3n}{2})}[w]_{A_3^{\rho,\,\loc}(\rn)}^{1/2}
+2^{j(1+\frac{3n}{2})}[w]_{A_3^{\rho,\,\loc}(\rn)}^{1/2}\r\}\\
&\hs\times\|f\|_{L^3(w)}\\
&\ls[w]_{A_3^{\rho,\,\tz}(\rn)}^{1/2}\|f\|_{L^3(w)}.
\end{align*}
This, together with Lemma \ref{lem-4} and the fact that $C_c^\fz(\rn)$ is dense in $L^p(w)$,
implies that \eqref{eq-3.6} holds true.

\emph{Step 3}.
Finally, by \eqref{eq-3.6} and the fact that, for any $f\in C_c^\fz(\rn)$ and $x\in\rn$,
$S_L(f)(x)\ls g_{L,\,\lz}^\ast(f)(x)$, we know that,
for any $w\in A_p^{\rho,\,\tz}(\rn)$ and $f\in C_c^\fz(\rn)$,
$$\|S_L(f)\|_{L^p(w)}\ls[w]_{A_p^{\rho,\,\tz}(\rn)}^{\max\{\frac12,\,\frac{1}{p-1}\}}\|f\|_{L^p(w)}.$$
This finishes the proof of Theorem \ref{thm-1}.
\end{proof}

Next, we prove Theorem \ref{thm-2}. To this end, we first introduce the following
lemma, which is a corollary of Lemma \ref{lem-10}.
In what follows, for any measurable function $w$
and any measurable subset $E$ of $\rn$, we write
$w(E):=\int_E w(y)\,dy$.
\begin{lemma}\label{lem-11}
Let $p\in(1,\,\fz)$, $1/p+1/p'=1$, $\tz\in[0,\,\fz)$, $n\geq 3$, $\rho$ be as in \eqref{eq aux} with $V\in RH_q(\rn)$
and $q\in(n/2,\,\fz)$, and $w\in A_p^{\rho,\,\tz}(\rn)$.
Then there exists a positive constant $C$ such that,
for any $\lz\in[1,\,\fz)$ and ball $B(x_B,\,r_B)$ of $\rn$,
\begin{align*}
w(B(x_B,\,\lz r_B))\le C[w]_{A_p^{\rho,\,\tz}(\rn)}^{\frac{p}{p-1}}
\lz^{pn}\lf[1+\frac{\lz r_B}{\rho(x_B)}\r]^{p'p\tz}w(B(x_B,\,r_B)).
\end{align*}
\end{lemma}
The proof of Lemma \ref{lem-11} is quite similar to
the classical one (see, for instance, \cite[p.\,133]{d00})
and we omit the details.

\begin{proof}[Proof of Theorem \ref{thm-2}]
Let $\mu\in[0,\,\fz)$.
We first introduce an auxiliary function $\wz{S}_L^\mu$, which is defined by setting,
for any $f\in C_c^\fz(\rn)$ and $x\in\rn$,
\begin{align*}
\wz{S}_L^\mu(f)(x):=\lf\{\int_0^\fz\int_{|x-y|<t}\lf[1+\frac{t}{\rho(y)}\r]^\mu
\lf|t^2Le^{-t^2L}(f)(y)\r|^2\,\dydt\r\}^{1/2}.
\end{align*}
For $\wz{S}_L^\mu$, by Lemmas \ref{lem aux} and \ref{lem-8}, we conclude that,
for any $f\in C_c^\fz(\rn)$ and $x\in\rn$,
\begin{align*}
&\lf\{\int_0^{\rho(x)}\int_{|x-y|<t}\lf[1+\frac{t}{\rho(y)}\r]^\mu
\lf|t^2Le^{-t^2L}(f)(y)\r|^2\,\dydt\r\}^{1/2}\\
&\hs\ls\lf\{\int_0^{\rho(x)}\int_{|x-y|<t}\lf|t^2Le^{-t^2L}(f)(y)\r|^2\,\dydt\r\}^{1/2}
\end{align*}
and
\begin{align*}
&\lf\{\int_{\rho(x)}^\fz\int_{|x-y|<t}\lf[1+\frac{t}{\rho(y)}\r]^\mu
\lf|t^2Le^{-t^2L}(f)(y)\r|^2\,\dydt\r\}^{1/2}\\
&\hs\ls\lf\{\int_{\rho(x)}^\fz\int_{|x-y|<t}\lf|\int_\rn \lf[1+\frac{t}{\rho(y)}\r]^{-(N-\mu)}
e^{-c\frac{|y-z|^2}{t^2}}|f(z)|\,dz\r|^2\,\dydt\r\}^{1/2},
\end{align*}
where $N\in(0,\,\fz)$ is arbitrary.
Let $p\in(1,\,\fz)$ and $\tz\in[0,\,\fz)$.
By the above two inequalities and an argument similar to that used in \emph{Step 2} of
the proof of Theorem \ref{thm-1},
we find that, for any $w\in A_p^{\rho,\,\tz}(\rn)$ and $f\in L^p(w)$,
\begin{align}\label{eq-3.8}
\lf\|\wz{S}_L^\mu(f)\r\|_{L^p(w)}
\ls[w]_{A_p^{\rho,\,\tz}(\rn)}^{\max\{\frac12,\,\frac{1}{p-1}\}}\|f\|_{L^p(w)}.
\end{align}

Next, we show that there exists a positive constant $C$ such that,
for any $\az\in[1,\,\fz)$, $w\in A_p^{\rho,\,\tz}(\rn)$ and $f\in L^p(w)$,
\begin{align}\label{eq-3.7}
\lf\|S_{L,\,\az}(f)\r\|_{L^p(w)}
\le C\az^{n+2\tz}[w]_{A_p^{\rho,\,\tz}(\rn)}^{\max\{1,\,\frac{1}{p-1}\}}
\lf\|\wz{S}_L^{4\tz}(f)\r\|_{L^p(w)}.
\end{align}
Indeed, by the Fubini theorem, Lemma \ref{lem-11} and the fact that $\az\in[1,\,\fz)$,
we find that, for any $w\in A_{2}^{\rho,\,\tz}(\rn)$ and $f\in C_c^\fz(\rn)$,
\begin{align*}
\lf\|S_{L,\,\az}(f)\r\|_{L^2(w)}^2
&=\int_\rn\int_0^\fz\int_{|x-y|<\az t}\lf|t^2Le^{-t^2L}(f)(y)\r|^2\,\dydt\,w(x)\,dx\\
&=\int_\rn\int_0^\fz w(B(y,\,\az t))\lf|t^2Le^{-t^2L}(f)(y)\r|^2\,\dydt\\
&\ls\az^{2n}[w]_{A_{2}^{\rho,\,\tz}(\rn)}^{2}
\int_\rn\int_0^\fz \lf[1+\frac{\az t}{\rho(y)}\r]^{4\tz}w(B(y,\,t))
\lf|t^2Le^{-t^2L}(f)(y)\r|^2\,\dydt\\
&\ls\az^{2(n+2\tz)}[w]_{A_{2}^{\rho,\,\tz}(\rn)}^{2}\int_\rn\int_0^\fz w(B(y,\,t))
\lf[1+\frac{t}{\rho(y)}\r]^{4\tz}\lf|t^2Le^{-t^2L}(f)(y)\r|^2\,\dydt\\
&\sim \az^{2(n+2\tz)}[w]_{A_{2}^{\rho,\,\tz}(\rn)}^{2}\lf\|\wz{S}_L^{4\tz}(f)\r\|_{L^2(w)}^2.
\end{align*}
Choose $\mathcal{F}$ to be the family of all pairs
$(\wz{f},\,\wz{g}):=[S_{L,\,\az}(f),\,\az^{n+2\tz}\wz{S}_L^{4\tz}(f)]$ with $f\in C_c^\fz(\rn)$.
Then we obtain
\begin{align*}
\int_\rn\lf[\wz{f}(x)\r]^{2}\,dx
&=\int_\rn\lf[S_{L,\,\az}(f)(x)\r]^2w(x)\,dx\\
&\ls\az^{(n+2\tz)r_0}[w]_{A_{2}^{\rho,\,\tz}(\rn)}^{2}\int_\rn \lf[\wz{S}_L^{4\tz}(f)(x)\r]^2w(x)\,dx\\
&\sim[w]_{A_{2}^{\rho,\,\tz}(\rn)}^{2}\int_\rn\lf[\wz{g}(x)\r]^{2}w(x)\,dx.
\end{align*}
This, combined with Remark \ref{rem-1}(ii), implies that,
for any given $p\in(1,\,\fz)$ and for any
$w\in A_p^{\rho,\,\tz}(\rn)$ and $f\in C_c^\fz(\rn)$,
\begin{align*}
\int_\rn\lf[S_{L,\,\az}(f)(x)\r]^{p}w(x)\,dx
&\ls [w]_{A_{p}^{\rho,\,\tz}(\rn)}^{p\max\{1,\,\frac{1}{p-1}\}}
\int_\rn \az^{(n+2\tz)p}\lf[\wz{S}_{L}^{4\tz}(f)(x)\r]^{p}w(x)\,dx.
\end{align*}
This, together with the fact that $C_c^\fz(\rn)$ is dense in $L^p(w)$, implies that
\eqref{eq-3.7} holds true.

From \eqref{eq-3.7}, \eqref{eq-3.8}, \eqref{eq-2.1} and the fact that
$\lz\in(2[1+\frac{2\tz}{n}],\,\fz)$,
it follows that, for any $w\in A_p^{\rho,\,\tz}(\rn)$ and $f\in L^p(w)$,
\begin{align*}
\lf\|g_{L,\,\lambda}^\ast(f)\r\|_{L^p(w)}
&\ls\sum_{j=0}^\fz 2^{-j\lambda n/2}\lf\|S_{L,\,2^j}(f)\r\|_{L^p(w)}
\ls\sum_{j=0}^\fz 2^{-j\lambda n/2}2^{j(n+2\tz)}[w]_{A_{p}^{\rho,\,\tz}(\rn)}^{\max\{1,\,\frac{1}{p-1}\}}
\lf\|\wz{S}_L^{4\tz}(f)\r\|_{L^p(w)}\\
&\ls\sum_{j=0}^\fz 2^{-j\frac{n}{2}[\lz-2(1+\frac{2\tz}{n})]}
[w]_{A_p^{\rho,\,\tz}(\rn)}^{\max\{1,\,\frac{1}{p-1}\}+\max\{\frac12,\,\frac{1}{p-1}\}}\|f\|_{L^p(w)}\noz\\
&\ls[w]_{A_p^{\rho,\,\tz}(\rn)}^{\max\{1,\,\frac{1}{p-1}\}+\max\{\frac12,\,\frac{1}{p-1}\}}\|f\|_{L^p(w)}.
\end{align*}
This finishes the proof of Theorem \ref{thm-2}.
\end{proof}

\noindent{\bf Acknowledgements.}
Junqiang Zhang would like to thank Yangyang Zhang
for some helpful discussions on the proof of Theorem \ref{thm-2}
and Dachun Yang would like to thank Professor Ji Li for a helpful
discussion on the subject of this article.


\bigskip

\noindent {Junqiang Zhang}

\medskip

\noindent{School of Science, China University of Mining and Technology-Beijing,
Beijing 100083, People's Republic of China}

\smallskip

\noindent {\it E-mail}: \texttt{jqzhang@cumtb.edu.cn} (J. Zhang)

\bigskip

\noindent {Dachun Yang (Corresponding author)}

\medskip

\noindent{Laboratory of Mathematics and Complex Systems (Ministry of Education of China),
School of  Mathematical Sciences, Beijing Normal University, Beijing 100875, People¡¯s Republic of China}

\smallskip

\noindent {\it E-mail}: \texttt{dcyang@bnu.edu.cn} (D. Yang)

\end{document}